\documentclass[12pt]{amsart}

\usepackage{etoolbox}
\usepackage{enumerate}

\makeatletter
\let\old@tocline\@tocline
\let\section@tocline\@tocline
\newcommand{\subsection@dotsep}{4.5}
\newcommand{\subsubsection@dotsep}{4.5}
\patchcmd{\@tocline}
  {\hfil}
  {\nobreak
     \leaders\hbox{$\m@th
        \mkern \subsection@dotsep mu\hbox{.}\mkern \subsection@dotsep mu$}\hfill
     \nobreak}{}{}
\let\subsection@tocline\@tocline
\let\@tocline\old@tocline

\patchcmd{\@tocline}
  {\hfil}
  {\nobreak
     \leaders\hbox{$\m@th
        \mkern \subsubsection@dotsep mu\hbox{.}\mkern \subsubsection@dotsep mu$}\hfill
     \nobreak}{}{}
\let\subsubsection@tocline\@tocline
\let\@tocline\old@tocline

\let\old@l@subsection\l@subsection
\let\old@l@subsubsection\l@subsubsection

\def\@tocwriteb#1#2#3{%
  \begingroup
    \@xp\def\csname #2@tocline\endcsname##1##2##3##4##5##6{%
      \ifnum##1>\c@tocdepth
      \else \sbox\z@{##5\let\indentlabel\@tochangmeasure##6}\fi}%
    \csname l@#2\endcsname{#1{\csname#2name\endcsname}{\@secnumber}{}}%
  \endgroup
  \addcontentsline{toc}{#2}%
    {\protect#1{\csname#2name\endcsname}{\@secnumber}{#3}}}%

\newlength{\@tocsectionindent}
\newlength{\@tocsubsectionindent}
\newlength{\@tocsubsubsectionindent}
\newlength{\@tocsectionnumwidth}
\newlength{\@tocsubsectionnumwidth}
\newlength{\@tocsubsubsectionnumwidth}
\newcommand{\settocsectionnumwidth}[1]{\setlength{\@tocsectionnumwidth}{#1}}
\newcommand{\settocsubsectionnumwidth}[1]{\setlength{\@tocsubsectionnumwidth}{#1}}
\newcommand{\settocsubsubsectionnumwidth}[1]{\setlength{\@tocsubsubsectionnumwidth}{#1}}
\newcommand{\settocsectionindent}[1]{\setlength{\@tocsectionindent}{#1}}
\newcommand{\settocsubsectionindent}[1]{\setlength{\@tocsubsectionindent}{#1}}
\newcommand{\settocsubsubsectionindent}[1]{\setlength{\@tocsubsubsectionindent}{#1}}

\renewcommand{\l@section}{\section@tocline{1}{\@tocsectionvskip}{\@tocsectionindent}{}{\@tocsectionformat}}%
\renewcommand{\l@subsection}{\subsection@tocline{1}{\@tocsubsectionvskip}{\@tocsubsectionindent}{}{\@tocsubsectionformat}}%
\renewcommand{\l@subsubsection}{\subsubsection@tocline{1}{\@tocsubsubsectionvskip}{\@tocsubsubsectionindent}{}{\@tocsubsubsectionformat}}%
\newcommand{\@tocsectionformat}{}
\newcommand{\@tocsubsectionformat}{}
\newcommand{\@tocsubsubsectionformat}{}
\expandafter\def\csname toc@1format\endcsname{\@tocsectionformat}
\expandafter\def\csname toc@2format\endcsname{\@tocsubsectionformat}
\expandafter\def\csname toc@3format\endcsname{\@tocsubsubsectionformat}
\newcommand{\settocsectionformat}[1]{\renewcommand{\@tocsectionformat}{#1}}
\newcommand{\settocsubsectionformat}[1]{\renewcommand{\@tocsubsectionformat}{#1}}
\newcommand{\settocsubsubsectionformat}[1]{\renewcommand{\@tocsubsubsectionformat}{#1}}
\newlength{\@tocsectionvskip}
\newcommand{\settocsectionvskip}[1]{\setlength{\@tocsectionvskip}{#1}}
\newlength{\@tocsubsectionvskip}
\newcommand{\settocsubsectionvskip}[1]{\setlength{\@tocsubsectionvskip}{#1}}
\newlength{\@tocsubsubsectionvskip}
\newcommand{\settocsubsubsectionvskip}[1]{\setlength{\@tocsubsubsectionvskip}{#1}}

\patchcmd{\tocsection}{\indentlabel}{\makebox[\@tocsectionnumwidth][l]}{}{}
\patchcmd{\tocsubsection}{\indentlabel}{\makebox[\@tocsubsectionnumwidth][l]}{}{}
\patchcmd{\tocsubsubsection}{\indentlabel}{\makebox[\@tocsubsubsectionnumwidth][l]}{}{}

\newcommand{\@sectypepnumformat}{}
\renewcommand{\contentsline}[1]{%
  \expandafter\let\expandafter\@sectypepnumformat\csname @toc#1pnumformat\endcsname%
  \csname l@#1\endcsname}
\newcommand{\@tocsectionpnumformat}{}
\newcommand{\@tocsubsectionpnumformat}{}
\newcommand{\@tocsubsubsectionpnumformat}{}
\newcommand{\setsectionpnumformat}[1]{\renewcommand{\@tocsectionpnumformat}{#1}}
\newcommand{\setsubsectionpnumformat}[1]{\renewcommand{\@tocsubsectionpnumformat}{#1}}
\newcommand{\setsubsubsectionpnumformat}[1]{\renewcommand{\@tocsubsubsectionpnumformat}{#1}}
\renewcommand{\@tocpagenum}[1]{%
  \hfill {\mdseries\@sectypepnumformat #1}}

\let\oldappendix\appendix
\renewcommand{\appendix}{%
  \leavevmode\oldappendix%
  \addtocontents{toc}{%
    \protect\settowidth{\protect\@tocsectionnumwidth}{\protect\@tocsectionformat\sectionname\space}%
    \protect\addtolength{\protect\@tocsectionnumwidth}{2em}}%
}
\makeatother

\makeatletter
\settocsectionnumwidth{2em}
\settocsubsectionnumwidth{2.5em}
\settocsubsubsectionnumwidth{3em}
\settocsectionindent{1pc}%
\settocsubsectionindent{\dimexpr\@tocsectionindent+\@tocsectionnumwidth}%
\settocsubsubsectionindent{\dimexpr\@tocsubsectionindent+\@tocsubsectionnumwidth}%
\makeatother

\settocsectionvskip{10pt}
\settocsubsectionvskip{0pt}
\settocsubsubsectionvskip{0pt}

\settocsectionformat{\scshape}
\settocsubsectionformat{\mdseries}
\settocsubsubsectionformat{\mdseries}
\setsectionpnumformat{\scshape}
\setsubsectionpnumformat{\mdseries}
\setsubsubsectionpnumformat{\mdseries}

\let\oldtableofcontents\tableofcontents
\renewcommand{\tableofcontents}{%
  \vspace*{-\linespacing}
  \oldtableofcontents}

\usepackage{graphicx}

\usepackage[bookmarks=true, bookmarksopen=true,%
bookmarksdepth=3,bookmarksopenlevel=2,%
colorlinks=true,%
linkcolor=blue,%
citecolor=blue,%
filecolor=blue,%
menucolor=blue,%
urlcolor=blue]{hyperref}

\usepackage{graphicx}
\usepackage{amssymb}
\usepackage{amsmath, amscd}

\newtheorem{thm}{Theorem}[section]

\newtheorem{theorem}[thm]{Theorem}
\newtheorem*{theorem*}{Theorem}
\newtheorem{proposition}[thm]{Proposition}
\newtheorem{corollary}[thm]{Corollary}
\newtheorem{lemma}[thm]{Lemma}

\theoremstyle{definition}
\newtheorem{definition}[thm]{Definition}

\theoremstyle{remark}
\newtheorem{remark}[thm]{Remark}

\numberwithin{equation}{section}

\newcommand{\R}{{\mathbb{R}}}
\newcommand{\C}{{\mathbb{C}}}
\newcommand{\N}{{\mathbb{N}}}

\newcommand{\Z}{{\mathbb{Z}}}

\newcommand\cO{\mathcal{O}}

\newcommand{\End}{\operatorname{End}}

\newcommand{\pa}{\partial}

\newcommand{\ind}{\operatorname{index}}

\usepackage{xspace}

\usepackage[margin=1in,marginparwidth=0.8in, marginparsep=0.1in]{geometry}

\begin{document}

\title[Ghost bubble censorship]{Ghost bubble censorship}
\author{Tobias Ekholm}
\address{  Department of mathematics, Uppsala University, Box 480, 751 06 Uppsala, Sweden \and
Institut Mittag-Leffler, Aurav 17, 182 60 Djursholm, Sweden}
\email{tobias.ekholm@math.uu.se}
\author{Vivek Shende}
\address{  
Center for Quantum Mathematics, Syddansk Univ., Campusvej 55
5230 Odense Denmark \and 
Department of mathematics, UC Berkeley, 970 Evans Hall,
Berkeley CA 94720 USA}
\email{vivek@math.berkeley.edu}

\thanks{TE is supported by the
	Knut and Alice Wallenberg Foundation as a Wallenberg scholar KAW2020.0307 and by the Swedish Research Council VR2020-04535. \\ \indent VS is 
	supported by the Villum Fonden (Villum Investigator 37814),  the Danish National Research foundation (DNRF157), 
	the Novo Nordisk Foundation (NNF20OC0066298), and the USA NSF (CAREER DMS-1654545).}
	
\maketitle

\begin{abstract}
When a Gromov limit of embedded holomorphic curves is constant on some component of the domain, the non-collapsed component must exhibit some degenerate behavior at the attaching points, such as high  multiplicity or vanishing of the holomorphic derivative. Here we show the same holds for maps which are only approximately $J$-holomorphic. 
\end{abstract}

\thispagestyle{empty}

{\small \tableofcontents}


\section{Introduction}\label{sec:intr} 
The stable map compactification of the moduli space of $J$-holomorphic curves contains many elements which do not arise from limits of maps from smooth holomorphic curves. This is true in particular for certain stable maps which send some irreducible components to points.  We call such components {\em ghost bubbles}, and call a stable $J$-holomorphic map {\em bare} when it has no ghost bubbles.  We are interested in when limits of smooth curves, or more generally, of  bare curves, may have ghost bubbles.

Throughout we fix a symplectic manifold $(X, \omega)$, a smooth Lagrangian submanifold $L\subset X$, which may be empty, and an almost complex structure $J$ such that $L$ is locally the fixed locus of an anti-holomorphic involution 
(such $J$ are plentiful). We consider maps from Riemann surfaces into $X$; for domains with boundary we impose Lagrangian boundary conditions in $L$.  

For $J$-holomorphic maps, the following result connects the formation of ghost bubbles to the geometry of the positive area components where they are attached.
\begin{theorem}[\cite{SOB}, see also \cite{doan-walpulski-embedded}, and \cite{ionel-genus1, zinger-sharp, niu2016refined} for various predecessors]  
\label{jhol censorship}
Consider a sequence of bare $J$-holomorphic maps
$u_\alpha\colon (S_\alpha,\partial S_{\alpha}) \to (X,L)$, $\bar\partial_{J}u_{\alpha}=0$, which Gromov converges to $u\colon (S,\partial S) \to (X,L)$. 
Let $u_{\mathrm{bare}} \colon S_{\mathrm{bare}} \to X$ be the restriction to the locus where $u$ is non-constant, 
and let $S_0$ be a (nonempty) connected component of $\overline{S \setminus S_\mathrm{bare}}$. 
Then one of the following holds: 
\begin{enumerate}
\item There is a point $p \in S_\mathrm{bare} \cap S_0$ with $\partial_J  u_\mathrm{bare} (T_p S_\mathrm{bare}) = 0$.  
\item There are two distinct points $p, q \in S_\mathrm{bare} \cap S_0$ with $\partial_J u_\mathrm{bare} (T_p S_\mathrm{bare}) = \partial_J  u_\mathrm{bare}( T_q S_\mathrm{bare})$.
\item There are at least three distinct points in $S_\mathrm{bare} \cap S_0$. 
\end{enumerate}
\end{theorem}

Theorem \ref{jhol censorship} implies an estimate on the codimension in moduli of ghost bubble formation.  
For example, the vanishing of the complex linear derivative at a boundary point is a codimension $(n-1)$ condition, $n=\dim L$.  Thus when $\dim_\R(X)=2n \ge 6$, ghost bubbling does not appear in transversely cut out moduli spaces 
of dimension zero and one.  We refer to this phenomenon as {\em ghost bubble censorship}. 

In geometric situations where
transversality of moduli can be achieved by generic choice of $J$, 
Theorem \ref{jhol censorship} or its predecessors have led to interesting results
\cite{ionel-genus1, zinger-sharp, niu2016refined, SOB, doan-walpulski-embedded}.  
In general, however, moduli of $J$-holomorphic curves fail to be transversely cut out,
in particular due to the appearance of multiply covered curves. 
While multiply covered
$J$-holomorphic curves will appear for any choice of $J$, they can be avoided by perturbing the Cauchy-Riemann equation to 
$\bar \partial_J u = \lambda(u)$.  This is a common strategy \cite{RT,FOOO,HWZ-GW}, which requires showing that 
moduli of solutions to the perturbed equation retain some of the behavior of the $J$-holomorphic case, e.g., Gromov compactness. 
Our purpose in the present article is to go further, and ask for $\lambda$ such that solutions to $\bar\partial_{J}(u)=\lambda(u)$ themselves (rather than just their moduli) 
retain geometric properties of $J$-holomorphic curves. More precisely, we want converging sequences of solutions to satisfy some generalization of Theorem 
\ref{jhol censorship} so as to bound the codimension of ghost bubble formation in moduli.  

Our interest in this question stems from our work  showing 
that a certain count of holomorphic curves on a Lagrangian associated to a knot conormal 
recovers the HOMFLY-PT invariant of the knot \cite{SOB}.  
Theorem \ref{jhol censorship} was a key technical ingredient in the proof; it was applicable
because in the geometry relevant to the HOMFLY-PT invariant is restricted and multiple covers can be excluded for topological reasons.
However, we also axiomatized certain properties of a perturbation system which would suffice 
to set up a theory of open skein-valued holomorphic curve counting without any such geometric restrictions.
The axioms implicitly require that some version of  Theorem \ref{jhol censorship} should hold for
the solutions to the perturbed equation.  Assuming the existence of such perturbations,  
corresponding results characterizing the full count of open holomorphic curves
in such geometries in terms of colored HOMFLY-PT invariants  hold \cite{ekholm-shende-unknot, ekholm-shende-colored}.  
In terms of that project, the role of the present article is to show how to build perturbations 
whose solutions exhibit ghost bubble censorship.  In the upcoming \cite{bare} we will  organize 
such perturbations into a system satisfying the axioms of \cite{SOB} in order 
to complete the foundations needed for \cite{ekholm-shende-unknot, ekholm-shende-colored}.

To state our results, we first fix some notation. Most of our hypotheses are imposed (and our analytic 
work takes place) near punctures or in neck regions in the domain surface. 
We always denote the circle factor of these cylinders as: 
\begin{equation}\label{eq:notationIfirst}
	I=\R/\Z. 
\end{equation}

To discuss local behavior of Gromov convergence we use the following terminology:  

\begin{definition} \label{ndc} 
A \emph{nodal degeneration of cylinders}  in an $\epsilon$-ball $B(\epsilon)\subset\R^{2n}$ is given by the data of maps  
$$u_+\colon (-\infty, 0] \times I \to B(\epsilon), \quad u_-\colon [0, \infty)  \times I \to B(\epsilon), \quad u_\alpha\colon [- \rho_\alpha, \rho_\alpha ] \times I \to B(\epsilon),$$
where $\rho_\alpha \to \infty$ is some sequence, that satisfy the conditions: 
\begin{enumerate}	
\item The asymptotic constants are defined and agree: $\lim_{s \to -\infty} u_+(s, t)   = \lim_{s \to \infty} u_-(s,t)$
\item In shifted coordinates $u_{\alpha, \pm }(s, t) := u_\alpha(s \pm \rho_\alpha, t)$, we have
$u_{\alpha, \pm} \to u_\pm$ in $C^{1}$
on compact subsets.  
\item The formulas $u_{+}(s,t)=u_{+}^{\rm disk}(e^{2\pi(s+it)})$ and $u_-(s,t)=u_{-}^{\rm disk}(e^{-2\pi(s+it)})$ determine
$C^{1,\gamma}$ maps $u^{\rm disk}_{\pm}\colon D\to\R^{2n}$ for some $\gamma > 0$.   
\end{enumerate} 
We denote such a degeneration by $u_\alpha \rightsquigarrow (u_+, u_-)$. 
\end{definition} 

Convergence on compact sets imposes no constraints at the middle of the neck, 
e.g., on $u_{\alpha}|_{[-1, 1] \times I}$.  However, consider the Fourier expansion, 
\begin{equation} 
 u_\bullet(s, t)  =  \sum_k c_{\bullet, m}(s) e^{2 \pi i m  t} \qquad \qquad \qquad (\bullet = +, -, \alpha)
\end{equation} 
and note that $u_{\bullet} (s,t)$ is holomorphic if and only if $e^{-2 \pi m s} c_{\bullet, m}(s)$ are constants.  Thus in the holomorphic case, 
convergence on compact sets translates to a condition on these constants, which then does constrain convergence in 
the neck region.  It is easy to see that the following properties hold for holomorphic
nodal degenerations, see Proposition \ref{holomorphic ndc reasonable}; we axiomatize them in general: 

\begin{definition} \label{reasonable necks} 
We say a nodal degeneration of cylinders has {\em plus-reasonable necks} if 
\begin{align}\label{reasonable necks derivative}
			\partial_J u^{\mathrm{disk}}_+ (0) &= \lim_{\rho_\alpha \to \infty} e^{2 \pi (\rho_\alpha-1)} \cdot c_{\alpha, 1}(1)  	 
		\end{align}
		
		\begin{equation}\label{reasonable necks slow fourier} 
		c_{\alpha, 1}(0) = e^{-2\pi}c_{\alpha, 1}(1) + {\mbox{\tiny$\mathcal{O}$}}(e^{-2\pi\rho_{\alpha}}). 
		\end{equation}
		
		\begin{align}\label{reasonable necks first fourier}
		&\left\|(u_{\alpha}(s,t)  - c_{\alpha,0}(0) - c_{\alpha, 1}(0)e^{2\pi(s+it)})|_{[-1,0]\times I}\right\|_{C^1} =\\\notag
		&\qquad\mathcal{O}(\|\bar\partial_{J} u_\alpha|_{[-\rho_{\alpha},0]\times I}\|_{2})+e^{-2\pi\rho_{\alpha}} \bigg(\mathcal{O}(\|(u_\alpha(s,t)-c_{\alpha,0}(0))|_{[-\rho_{\alpha},-\rho_{\alpha}+1]\times I}\|_{1}) + 
		{\mbox{\tiny$\mathcal{O}$}}(1)  \bigg), 
		\end{align}
	where the implicit constants in the estimates should not depend on $\alpha$ or $\rho_\alpha$, and ${\mbox{\tiny$\mathcal{O}$}}(1)$ refers
	to the behavior as $\rho_\alpha \to \infty$. The norm $\|\cdot\|_{k}$ is the usual $H^k$ norm, see \eqref{eq: k-norm def}. 
\end{definition}

\begin{remark} \label{remark ghost minus convention} 
There is a corresponding notion of minus-reasonable necks; in fact when we later construct classes of reasonable neck degenerations, 
they will be both plus and minus reasonable. 
We will however always apply this notion where the $u_-$-end is the forming ghost bubble
(this is just a convention),
and for this it is the estimates stated in Definition \ref{reasonable necks} which are relevant. 
\end{remark} 

Consider  bare maps $u_\alpha\colon S_\alpha \to X$  Gromov converging to a map $u\colon S \to X$; in particular, there are local coordinates in which the maps form a nodal degeneration of cylinders as above, for more details see Section \ref{sec: grovergence notation}. Let $S^- \subset S$ be a connected component of the ghost locus. 
Choose nodal degeneration of cylinder 
cooordinates for each  
node attaching $S^-$ and $S \setminus S^-$ such that for the associated nodal degeneration of cylinders, 
$u_- \equiv 0$.  This determines a locus $S^-_{\alpha} \subset S_\alpha$ `limiting to $S^-$'; 
the intersection of $S^-_{\alpha}$ with the degenerating cylinder is $[-\rho_\alpha, 0] \times I$.  
We say the sequence $u_\alpha$ is {\em $J$-holomorphic near $S^-$} if, for all sufficiently large $\alpha$, 
\begin{equation}
\label{holomorphic ghosts}
\bar \partial_J u_\alpha|_{S^-_{\alpha}} = 0.
\end{equation}

We will show the following. 

\begin{theorem}\label{thm:compactness intro} {\rm (\ref{thm:compactness})}
	Consider a sequence of bare maps
	$u_\alpha\colon (S_\alpha,\partial S_{\alpha}) \to (X,L)$ which Gromov converges to a stable map $u\colon (S,\partial S) \to (X,L)$. 
    Let $u_\mathrm{bare} \colon S_\mathrm{bare} \to X$ be the restriction of $u$ to the components of positive symplectic area, 
	and let $S^-$ be a (nonempty) connected component of $\overline{S \setminus S_\mathrm{bare}}$. 
Assume the sequence $u_\alpha$ has plus-reasonable necks at the nodes connecting $S^-$ to $ S_\mathrm{bare}$, and 
	 is {\em $J$-holomorphic near $S^-$}.  
	Then one of the following holds: 
\begin{enumerate}
	\item There is a point $p \in S_\mathrm{bare} \cap S^-$ with $\partial_J  u_\mathrm{bare} (T_p S_\mathrm{bare}) = 0$.  
	\item There are two distinct points $p, q \in S_\mathrm{bare} \cap S^-$ with $\partial_J u_\mathrm{bare} (T_p S_\mathrm{bare}) = \partial_J  u_\mathrm{bare}( T_q S_\mathrm{bare})$.
	\item There are at least three distinct points in $S_\mathrm{bare} \cap S^-$. 
\end{enumerate}
\end{theorem}

Let us note the following subtlety in the statement:  we do not assume $\bar \partial_J u_\mathrm{bare} (T_p S_\mathrm{bare}) = 0$, 
so the conclusion $\partial_J  u_\mathrm{bare} (T_p S_\mathrm{bare}) = 0$ does not imply that $u_\mathrm{bare}$ is singular at $p$.  
Nevertheless, $\partial_J  u_\mathrm{bare} (T_p S_\mathrm{bare}) = 0$  is a non-generic phenomenon of codimension $(2n- 2)$, or $(n-1)$, for $p$ in the interior, or on the boundary, respectively, in moduli, where $2n=\dim X$. 

The idea of the proof is to apply the Riemann-Hurwitz theorem to the forming ghost bubble: its
Euler characteristic determines the winding numbers of the neck regions, which, on the other end of the neck, 
determines the complex linear derivative $\partial_J u$ at the attaching point on the non-collapsed component.
Thus, the main points are to show that Riemann-Hurwitz is applicable, that the winding numbers are well defined, and 
that having winding number one is characterized by nonvanishing of $\partial_J u$.  These properties are
ensured using monotonicity and the conditions in Definition \ref{reasonable necks}.   

We also note that, while Theorems \ref{thm:compactness intro} and \ref{jhol censorship} are stated for curves
with Lagrangian boundary conditions, the proofs reduce immediately to the case of closed curves.  Indeed, 
all arguments are local near the forming ghost bubble. Either the ghost bubble is away from $L$, in which case
we are locally in the closed case already, or it is on $L$, in which case, as we have assumed that $L$ is locally
the fixed locus of an anti-holomorphic involution, we may double the curve and again reduce to the closed case.  
Because of this we will not explicitly mention Lagrangian boundary conditions in the remainder of the text.

It is not immediately obvious that Theorem \ref{thm:compactness intro} generalizes Theorem \ref{jhol censorship}. 
To see that it does, one needs to know that ghost bubble formation for $J$-holomorphic curves can be described by
nodal degeneration of cylinders with reasonable necks.  This is a special case of our results below, 
see Remarks \ref{J-hol ghost reas proof} and \ref{J-hol reas proof}. 

More generally, 
we are interested in constructing perturbations $\lambda(u)$ such that any nodal degeneration satisfying 
$\bar \partial_J u = \lambda(u)$ will have reasonable necks and hence be subject to Theorem \ref{thm:compactness intro}.  
For applications, we would like to have a criterion expressed in functional analytic norms of $\lambda(u)$, 
or equivalently, $\bar \partial_J u$.  
Intuitively, the idea is that just as holomorphicity (hence constancy of Fourier coefficients)
allows convergence on compact sets to control degeneration in the neck, 
sufficiently strong estimates on $\bar \partial_J u$ should allow us to estimate the variation of Fourier coefficients
well enough to guarantee reasonable necks.  

These estimates will be expressed in terms of weighted Sobolev norms; we now fix notation. 
If $f\colon \R\times I\to\R^{2n}$, we denote the usual Sobolev norms of $f$ by 
\begin{equation}\label{eq: k-norm def}
\| f \|_k := \left( \int_{\R\times I} \sum_{j=0}^{k} |d^{j} f|^{2} \,dsdt\right)^{\frac12},
\end{equation}
and write $H^k$ to indicate the space of functions with bounded Sobolev norm. If $f$ is defined on $U\subset\R\times I$ we sometimes consider the restricted norm
$$\| f|_{U} \|_k := \left( \int_{U} \sum_{j=0}^{k} |d^{j} f|^{2} \,dsdt\right)^{\frac12}.$$
For $\delta\in\R$, we use the weighted Sobolev norm 
\begin{equation} \label{big at end}
\|f\|_{k,\delta}:=\left(\int_{ \R \times I}\left(\sum_{j=0}^{k} |d^{j} f|^{2}\right) e^{2 \delta |s|} dsdt\right)^{\frac12}.
\end{equation} 
For finite neck-regions we will use re-scaled versions of the weighted norms as follows. If $f\colon[-\rho,\rho]\times I \to\R^{2n}$ then
\begin{equation} \label{big in middle} 
\|f\|_{k,\delta}^{\;\wedge} := e^{\delta\rho}\|f|_{[-\rho,\rho]\times I}\|_{k, - \delta} =  \left(\int_{[-\rho, \rho] \times I}\left(\sum_{j=0}^{k} |d^{j} f|^{2}\right) e^{2 \delta (\rho - |s|)} \,dsdt\right)^{\frac12}.
\end{equation}

The weight function for \eqref{big at end} is large at the ends of the strip, whereas 
the weight function for \eqref{big in middle} is large at the middle of the strip.  This is natural in
the context of node formation because of the shift in the coordinates used in Definition \ref{ndc}.
We write $H^k_\delta$ and $\widehat H^{k}_\delta$ for the space of functions 
with norms $\|\,\cdot \|_{k, \delta}$ and $\|\cdot\|_{k,\delta}^{\;\wedge}$.   

For constants $\xi_{\pm}\in\R^{2n}$, 
we consider 
\begin{align*}
\xi_{+}\otimes e^{2\pi(s+it)}(ds-idt) \ &= \ \xi_{+}\otimes d\bar z \ \colon \ T(-\infty,0]\times I \ \to \ T\R^{2n},\\
-\xi_{-}\otimes e^{-2\pi(s+it)}(ds-idt) \ &= \ \xi_{-}\otimes d\bar z \ \colon \ T[0,\infty)\times I \ \to \ T\R^{2n}.
\end{align*}

\begin{definition} \label{admissible}
	A nodal degeneration of cylinders $u_\alpha \rightsquigarrow (u_+, u_-)$ 
	\begin{itemize}
			\item has \emph{exponential neck decay}  if, for some $\delta > \pi$,  
		\begin{equation} \label{neck decay}  
		\|u_\alpha - c_{\alpha,0}(0)\|^{\;\wedge}_{3, \delta} = \cO(1);
		\end{equation} 
		\item is \emph{$\bar\partial_{J}$-compatible} if, for some $\delta' > 0$ there exist 
		$\xi_{\alpha,\pm} \in \R^{2n}$ such that 
		\begin{align}
			\label{limit compatible} 
			&\|(\bar\partial_{J}u_{\alpha,+} \ - \ \xi_{\alpha,+}\otimes d\bar z)|_{[- \rho_\alpha+1, 0] \times I}\|_{0,2\pi + \delta'} \\\notag
			&\qquad\qquad+ 
			\|(\bar\partial_{J} u_{\alpha,-} \ - \ \xi_{\alpha,-}\otimes d\bar z) |_{[0, \rho_\alpha-1] \times I}\|_{0,2\pi + \delta'} = \mathcal{O}(1);
		\end{align}  
		\item satisfies the \emph{cut-off decay condition} if, 
		\begin{equation}
			\label{cut-off decay}
			\|\bar\partial_{J} u_{\alpha}|_{[-1,1]\times I}\|_{0} = {\mbox{\tiny$\mathcal{O}$}} (e^{-2\pi\rho_{\alpha}}).
		\end{equation}
	\end{itemize}
	We will say that a nodal degeneration of cylinders satisfying all the above hypothesis is \emph{admissible}. 
\end{definition}

\begin{remark} 
The nodal degeneration requirement that $u^{\rm disk}_{\pm} \in C^{1, \gamma}$ can itself be ensured by requiring that $u_{\pm}$ lie in appropriate
weighted Sobolev spaces, see Lemma \ref{r:4deltaC1eps}.  
\end{remark} 

We prove: 
\begin{theorem} \label{adm reas intro}  {\rm (\ref{admissible reasonable})} 
An admissible nodal degeneration of cylinders 
has plus-reasonable necks.
\end{theorem}
About the proof: the special case $J = J_{\mathrm{std}}$ can be checked by a straightforward study of Fourier expansions, in fact using only
\eqref{limit compatible} and \eqref{cut-off decay}.  
Indeed, Lemma \ref{lem: limiting cutoff decay standard} verifies \eqref{reasonable necks slow fourier},
 Proposition \ref{prp: dbarcompatible controls neck minus middle std} verifies \eqref{reasonable necks derivative},
 and
Corollary \ref{Jstd rnff} verifies \eqref{reasonable necks first fourier}. 
To generalize to the case
of arbitrary $J$, we use a Carleman similarity result for cylinders, Theorem \ref{thm: Carleman}, 
the hypotheses of which follow from \eqref{neck decay}. 
  
\begin{remark} \label{J-hol ghost reas proof} 
Conditions \eqref{limit compatible}, \eqref{cut-off decay} are tautologically satisfied for $J$-holomorphic curves. 
In fact, one can show fairly readily that for $J$-holomorphic ghost bubble formation, one can
choose the nodal degeneration cylinders to have exponential neck decay, see Corollary \ref{J hol neck decay}.
With this, we deduce
Theorem \ref{jhol censorship}  from Theorems \ref{thm:compactness intro} and \ref{adm reas intro}.  This 
is similar to the proof of Theorem \ref{jhol censorship} given in \cite{SOB}. 
\end{remark} 

It is easy to see that holomorphic curves have exponential neck decay, see Proposition \ref{holomorphic ndc reasonable}.  
More generally, we can show that exponential neck decay follows from 
a $2$-norm version of \eqref{limit compatible} and a 
$2$-norm version of \eqref{cut-off decay} with smaller exponential weight: 

\begin{theorem} \label{neck decay from 2 norms} {\rm (\ref{neck decay from 2 norms *})}
	A nodal degeneration of cylinders $u_\alpha \rightsquigarrow (u_+, u_-)$ satisfying  
		\begin{align}
			\label{limit compatible 2} 
			&\|(\bar\partial_{J}u_{\alpha,+} \ - \ \xi_{\alpha,+}\otimes d\bar z)|_{[- \rho_\alpha+1, 0] \times I}\|_{2,2\pi + \delta'} \\\notag
			&\qquad\quad+ 
			\|(\bar\partial_{J} u_{\alpha,-} \ - \ \xi_{\alpha,-}\otimes d\bar z) |_{[0, \rho_\alpha-1] \times I}\|_{2,2\pi + \delta'} = \mathcal{O}(1).
		\end{align}  
		and, for some $\delta > \pi$, 
		\begin{equation}
			\label{cut-off decay2}
			\|\bar\partial_{J} u_{\alpha}|_{[-1,1]\times I}\|_{2} = \mathcal{O}(e^{-\delta \rho_{\alpha}}).
		\end{equation}
	has exponential neck decay. 
\end{theorem} 

\begin{remark} \label{J-hol reas proof}
The hypotheses of Theorem \ref{neck decay from 2 norms} are tautologically satisfied for $J$-holomorphic curves, so 
any $J$-holomorphic nodal degeneration of cylinders has exponential neck decay, hence is admissible.  
Thus
Theorem \ref{jhol censorship} follows also from Theorems \ref{thm:compactness intro}, \ref{adm reas intro}, and \ref{neck decay from 2 norms}, without using Corollary \ref{J hol neck decay}.
\end{remark} 

Finally, in Section \ref{hwz admissible} we show how to construct local perturbations $\lambda$, naturally adapted to the setting of \cite{HWZ-GW}, such that nodal degenerations of solutions to the perturbed $\bar \partial_J$-equation, $\bar\partial_{J}u=\lambda(u)$ satisfy \eqref{cut-off decay}, \eqref{limit compatible 2}, and \eqref{cut-off decay2}.
So by Theorems \ref{adm reas intro} and \ref{neck decay from 2 norms}, we see that Theorem \ref{thm:compactness intro} applies to such degenerations.

\section{Cylinders, disks, and Fourier coefficients} 

Here we fix notation for cylinder and disk parameterizations, 
and discuss basic relations between them. Consider the two parameterizations of the punctured unit disk
$$D^\circ =  \{z\in\C\colon 0 < |z| \le 1\}$$ 
\begin{alignat}{4}\label{eq:exppolar}
\eta_- : \,\, & [0, \infty) &&\times \R / \Z  \to  D^\circ,  \qquad\qquad & \eta_+: \,\, & (-\infty, 0] &&\times \R / \Z\to  D^\circ,  \\\notag
&(s, t)  &&\mapsto \quad e^{- 2 \pi (s + i t)}   \qquad\qquad    &&(s, t) &&\mapsto \quad e^{2 \pi (s + i t)}. 
\end{alignat} 
Given maps  
$$u_-\colon [0, \infty)  \times I \to \R^{2n},  \qquad u_+\colon (-\infty, 0] \times I \to \R^{2n}, $$
we define $ u_\pm^{\mathrm{disk}} \colon D^\circ \to \R^{2n}$ by the change of variables 
\begin{equation}
u_\pm =  u_\pm^{\mathrm{disk}} \circ \eta_\pm.
\end{equation}

We Fourier expand: 
\begin{equation} \label{eq: fourier}
 u_\bullet(s, t)  =  \sum_k c_{\bullet, m}(s) e^{2 \pi i m  t} \qquad \qquad \qquad (\bullet = +, -, \alpha),
\end{equation} 
where $c_{\bullet, m}(s)$ are $\R^{2n}$-valued Fourier coefficients.
We have 
$$
\bar \partial  u_\bullet(s, t) = \tfrac12\sum_m (c_{\bullet, m}'(s) - 2\pi m c_{\bullet, m}(s)) e^{2 \pi  i m t}.
$$ 
In particular,
\begin{equation} \label{eq: holomorphic fourier} 
\bar \partial  u_\bullet = 0 \quad \text{ if and only if }\quad c_{\bullet, m}(s)= e^{2\pi m s} \cdot c_{\bullet,m},  \text{ where $c_{\bullet, m}$ is constant.} 
\end{equation}
If $u_{\alpha}\rightsquigarrow(u_{+},u_{-})$, see Definition \ref{ndc}, then, for any fixed $\rho_{0}>0$, the Fourier coefficients 
$$c_{+, m}\colon (-\infty, 0] \to \R^{2n},  \qquad c_{-, m}\colon [0, \infty) \to \R^{2n},\qquad 
c_{\alpha, m}\colon [-\rho_\alpha, \rho_\alpha] \to \R^{2n},$$  
satisfy
\begin{align}
\label{eq: gromov convergence fourier plus}
&c_{\alpha, m}(s + \rho_\alpha) \to c_{+, m}(s),\qquad s\in [-\rho_{0},0],\\
\label{eq: gromov convergence fourier minus}
&c_{\alpha, m}(s - \rho_\alpha) \to c_{-, m}(s),\qquad s\in [0,\rho_{0}].
\end{align}

The next result is straightforward consequence of Taylor expansion, we include it for future reference.  

\begin{lemma} \label{lemma: s limiting fourier coefficient is derivative} 
Let $J$ be an almost complex structure on $\R^{2n}$, standard at $0$. Consider maps $u_{-}\colon[0,\infty)\times I\to\R^{2n}$, and $u_{+}\colon (-\infty,0]\times I\to\R^{2n}$, and assume that the corresponding maps  
$u_{\pm}^{\mathrm{disk}} \colon D\to\R^{2n}$ are $C^{1,\gamma}$ at the origin, for some $\gamma > 0$. Then, with notation as in Equation \eqref{eq: fourier}, 
$$\partial_J u^{\mathrm{disk}}_+ (0) = e^{-2\pi s} c_{+, 1}(s)  + \mathcal{O}(e^{2\pi \gamma s}) \qquad \qquad (s \le 0)$$
$$ \partial_J  u^{\mathrm{disk}}_- (0)  = e^{2 \pi s} c_{-, -1}(s) + \mathcal{O}(e^{-2\pi\gamma s}) \qquad \qquad (s \ge 0)$$
\end{lemma}
\begin{proof}
	We discuss $u_+$, the case of $u_-$ is similar. 
	Consider the Taylor expansion around $0$: 
	\[ 
	 u^{\mathrm{disk}}_+(z)= f_{z} z + f_{\bar z} \bar z  + \mathcal{O}(|z|^{1+\gamma}).
	\] 
	Changing variables $z=e^{2\pi(s+it)}$, $s+it\in[0,-\infty)\times I$ we find
	\[ 
	u_+(s,t)= f_{z} e^{2\pi(s+it)} + f_{\bar z} e^{2\pi(s-it)}  + \mathcal{O}(e^{(2\pi +\gamma)s}).
	\]
	Thus,
	\begin{equation}\label{eq:close to linear} 
		e^{-2\pi s}c_{+, 1}(s)=e^{-2\pi s}\int_{I} u_{+}(s,t)e^{-2\pi it} dt = f_{z} + \mathcal{O}(e^{2\pi\gamma s}).
	\end{equation}
	The lemma follows.
\end{proof}

We record (but use nowhere in this article) the following sufficient condition for the H\"older continuity required in Lemma \ref{lemma: s limiting fourier coefficient is derivative}. 
For simpler notation let $H^{k}_{\delta}(+)=H^{k}_{\delta}((-\infty,0]\times I,\R^{2n})$ and 
$H^{k}_{\delta}(-)=H^{k}_{\delta}([0,\infty)\times I,\R^{2n})$.
\begin{lemma}\label{r:4deltaC1eps}    
If $u_{\pm} \in H^{3}_\delta(\pm)$ for some $\delta > 0$, then $u^{\mathrm{disk}}_{\pm} \in C^{0, \gamma}$ for any $0 < \gamma < \frac{1}{2\pi}\delta$ and, similarly, if $du_{\pm}\in H^{3}_{2\pi+\delta}(\pm)$ then $du^{\mathrm{disk}}_{\pm} \in C^{0, \gamma}$. In particular, if there are constants $f_{\pm,0},f_{\pm,1},f_{\pm,-1}\in\R^{2n}$ such that $\|u_{\pm}-f_{\pm,0}\|_{3,\delta}=\mathcal{O}(1)$ and $\|du_{\pm} - f_{\pm,1}\otimes dz - f_{\pm,-1}\otimes d\bar z\|_{3,2\pi+\delta}=\mathcal{O}(1)$ then $u_{\pm}^{\rm disk}\in C^{1,\gamma}$.
\end{lemma}
\begin{proof}
	By definition of the weighted 3-norm $u_{\pm}(s, t)\cdot e^{\delta |s|} \in H^3$. Since the 3-norm 
	controls the $C^{1}$-norm we find $|D(u_{\pm}(s, t)\cdot e^{\delta |s|})|<C$. 
	
	Let $0<\eta<\delta$. Then for all $|s|$ sufficiently 
	large, if $|u_{\pm}(s, t)\cdot e^{\delta |s|}|> e^{\eta|s|} $ then by the bound of the derivative, $|u_\pm(s',t)\cdot e^{\delta |s'|}|> e^{\eta |s|}-C$ for $s-1\le s'\le s+1$, and   
	\[ 
	\int_{[s,s+1]\times I}|u_{\pm}(s,t)|^{2}e^{2\delta|s|}\,dsdt \ \ge \  e^{2\eta |s|} - C \ \to \infty, \quad\text{as}\quad s\to\infty.
	\] 
	This contradicts $\|u_{\pm}\|_{0,\delta}<\infty$ and we conclude that for every $\epsilon\in (0,\delta)$, $|u_{\pm}(s,t)\cdot e^{\delta s}|=\mathcal{O}(e^{\eta s})$ as $s\to\infty$. Hence
	\[ 
	|u^{\mathrm{disk}}(z)|=\mathcal{O}\left(|z|^{\frac{1}{2\pi}(\delta-\eta)}\right),
	\]
	and thus $u^{\mathrm{disk}}\in C^{0, \gamma}$ for any $\gamma<\frac{1}{2\pi}\delta$.
	
	To see the statement for the derivative note that $e^{\pm 2\pi(s+it)}du_{\pm}=du_{\pm}^{\rm disk}$ and apply the same argument. For the last statement observe that we can take
	\[ 
	u_{\pm}^{\rm disk}(e^{\mp 2\pi(s+it)})=f_{\pm,0}+f_{\pm,1}e^{\mp 2\pi(s+it)}+f_{\pm,-1}e^{\mp 2\pi(s-it)} + {\tilde u}^{\rm disk}_{\pm}(e^{\mp 2\pi(s+it)}),
	\]
	where $\tilde u$ satisfies the necessary bounds on the norm.
\end{proof}

\section{Ghost bubble censorship assuming reasonable necks} \label{sec: gbc} 
In this section we introduce basic tools in the study of holomorphic maps and then give a proof of Theorem \ref{thm:compactness intro} \eqref{thm:compactness} under the assumption that the degenerating sequence of cylinders have reasonable necks.  

\subsection{Notation for Gromov convergence} \label{sec: grovergence notation}
We will always require maps from
smooth domains to be $C^1$ (sometimes implicitly by requiring $H^3_{\mathrm{loc}}$). 
We fix two parameterizations of the punctured unit disk. Consider a point $p \in S$ where $S$ is a Riemann surface.  By \emph{holomorphic polar coordinates around $p$}
we mean the choice of a holomorphic map $(D, 0) \to (S, p)$ which is composed with one of the parameterizations in \eqref{eq:exppolar} 
to give a map $[0, \infty) \times I \to S \setminus p$, or 
$(-\infty, 0] \times I \to S \setminus p$. 

Now fix some curve $S$,
and a sequence $S_{\alpha}$ converging in Deligne-Mumford space,\footnote{To  describe stable maps from unstable domains, 
we  fix stabilizing points in the domain.} $S_\alpha \to S$ as $\alpha \to \infty$.  
A node $\zeta$ of $S$ is specified by two points on the underlying smooth domains; order them (for the moment arbitrarily)
$\zeta=\{\zeta^+,\zeta^-\}$. 
Fix holomorphic polar coordinates identifying $R_+(\zeta) = (-\infty, 0] \times I$  and
$R_-(\zeta)= [0, \infty) \times I$  with punctured closed disks around $\zeta^+$ and $\zeta^-$.  

Since $S_{\alpha}\to S$, there are gluing distances $\rho_\alpha(\zeta)$, with $\rho_\alpha(\zeta) \to \infty$ as $\alpha \to \infty$, 
such that if we take 
$$[\rho_{\alpha}(\zeta),\infty)\times I \ =: \  U_\alpha(\zeta^-) \ \subset \ R_-(\zeta) \ = \ [0,\infty) \times I$$
(and similarly for $\zeta^-$)
then we may fix the inclusions
\begin{equation} \label{eq: trivialization near node}
  S \setminus \left(\bigcup_{\zeta}  U_{\alpha}(\zeta^+) \cup U_{\alpha}(\zeta^-)\right) \ =: \ S_{\alpha}^\circ \ \hookrightarrow \ S_\alpha
\end{equation}
which extend to gluings $\overline{S_{\alpha}^\circ} \twoheadrightarrow  S_\alpha$.   

Along the component of $\partial \overline{S_{\alpha}^\circ}$ associated to a given node $\zeta$,
one has $[0, \rho_\alpha(\zeta)] \times I$ and $[-\rho_\alpha(\zeta), 0] \times I$ which are glued along the $\pm \rho_\alpha(\zeta) \times I$ ends.  
We write $R_\alpha(\zeta)$ for the resulting `gluing region' associated to $\zeta$:  
\begin{equation}\label{eq:defgluingregR} 
R_{\alpha}(\zeta)= [-\rho_{\alpha}(\zeta),\rho_{\alpha}(\zeta)]\times I \subset S_{\alpha}.
\end{equation}

\begin{remark}  \label{rem: shift}
Note  that $[-\rho_\alpha, 0] \times I \subset R_\alpha(\zeta)$ is identified with the original $[0, \rho_\alpha] \subset R_-(\zeta)$,
the identification being a shift by $\rho_\alpha$.  Similarly,  $[0, \rho_\alpha] \times I \subset R_\alpha(\zeta)$ is identified
with $[0, -\rho_\alpha] \times I \subset R_+(\zeta)$. 
In particular, exiting $R_\alpha$ out the $\rho_\alpha$ end goes to the region near
the $\zeta^+$ side of the node, and exiting out the $-\rho_\alpha$ end goes to the region near the $\zeta^-$ side.  
\end{remark}

\begin{definition}
We say that such $(u_\alpha, S_\alpha)$ Gromov converges to $(u, S)$ if for each (newly formed) node $\zeta \in S$ there exist
coordinates $R_\alpha(\zeta)$ and $R_{\pm}(\zeta)$ as above such that 
$u|_{R_\alpha(\zeta)} \rightsquigarrow (u|_{R_+(\zeta)}, u|_{R_-(\zeta)})$ is a
nodal degeneration of cylinders  in the sense of Definition \ref{ndc}, 
and $u_\alpha \to u$ in the complement of the union of $R_{\pm}(\zeta)$ regions. 
\end{definition} 

If we fix a decomposition of the limiting domain $S$ as $S = S^+ \cup S^-$ with $S^+ \cap S^-$ all newly formed nodes, 
then for each such node $\zeta$ we choose $\zeta^+$ to be the end in $S^+$ and $\zeta^-$ the end in $S^-$.  
In this case we have a corresponding disjoint union decomposition 
$S_\alpha^\circ = S_\alpha^+ \sqcup S_\alpha^-$. 

In particular, if the $S_\alpha$ are bare and we are studying 
a connected component of the ghost locus, we may and will always choose it to be $S^-$. 
This is the convention alluded to in Remark \ref{remark ghost minus convention} above. 

\subsection{Area estimates}\label{appendix1}
When $J$ is standard, and we have a holomorphic map $v\colon (S, \partial S) \to (\R^{2n},\R^{n})$, we have the formula 
$$
\mathrm{Area}(v) = \|dv\|^2 = \int_S v^* \omega_{\mathrm{std}} = \int_{\partial S} v^*( x dy ). 
$$
Here $\omega_{\mathrm{std}} = \sum dx_i \wedge dy_i$, $xdy := \sum x_i dy_i$, and the boundary of $S$ is oriented according to the outward normal first rule. 
Thus, the area of the curve is controlled by the action along the boundary. In particular there are no non-constant 
closed holomorphic curves without boundary.

Consider now the more general situation of an almost complex structure $J$ on $\R^{2n}$, and an arbitrary map.  We can
still control area by action, with a correction for $\bar \partial_J$.  

\begin{lemma} \cite[Section 2.3]{sikorav}\label{l:Sikorav}   
Fix an almost complex structure $J$ on $\R^{2n}$ with $J(0)=J_{\rm std}$ and let $B(0;\epsilon)\subset\R^{2n}$ be the $\epsilon$-ball around the origin. For all sufficiently small $\epsilon$, the following estimate holds for any 
map $u\colon (S, \partial S) \to (B(0;\epsilon),B(0;\epsilon)\cap\R^{n})$ from a Riemann surface,   
\[ 
\|du\|^{2}\le 3\|\bar \partial_J u \|^{2} + 2\int_{\partial S}u^{\ast}(xdy),
\]
where the boundary is oriented according to the outward normal first rule.
\end{lemma} 

\begin{proof}
Since $J$ is smooth we have $\|J-J(0)\|_{C^{k}}=\mathcal{O}(\epsilon)$ on $B(0;\epsilon)$, for all $k$. 
Rewrite, following \cite{sikorav}, the perturbed Cauchy-Riemann equation in terms of the standard complex structure $J(0)=J_{\mathrm{std}}$:
\[ 
\bar\partial_{J} u= \bar\pa u - q(z)\partial u,
\]
where $q(z)=(J(u(z))+J_0)^{-1}(J(u(z))-J_0)$. Then $|q|=\mathcal{O}(\epsilon)$ and we find for all sufficiently small $\epsilon>0$ that  
\begin{align*}
|\partial u|^{2}+|\bar\partial u|^{2}
& \le |\partial u|^{2}  + |q|^2 |\partial u|^2 + |\bar \partial_J u|^2  \\ 
& \le |\partial u|^{2} + (1-2|q|^{2})|\partial u|^{2} +|\bar\partial_{J} u|^{2} \\
& =
2(|\partial u|^{2} -|q|^{2}|\partial u|^{2}-|\bar\partial_{J} u|^{2})  +3|\bar\partial_{J} u|^{2}\\
&\le 2(|\partial u|^{2}-|\bar\partial u|^{2}) +3|\bar\partial_{J} u|^{2},
\end{align*}
where the first and final line use triangle inequality, and the second uses $C\epsilon^{2}\le 1-2 C\epsilon^{2}$ for all sufficiently small $\epsilon>0$.
Integrating over $S$ then gives
\[ 
\|du\|^{2}\le 2\int_{S} u^{\ast}(dx \wedge dy) +3\|\bar\partial_{J} u\|^{2}=2\int_{\partial S}u^{\ast}(xdy) +3\|\bar\partial_{J} u\|^{2}.
\]	
\end{proof}

We have the following consequence of monotonicity.  

\begin{lemma}\label{lem: monotonicity}
	Let $J$ be a tame almost complex structure on $\R^{2n}$ (i.e., $\omega_{\rm std}(J\cdot,\cdot)$ is uniformly positive). Then there exists $\epsilon>0$ and $C>0$ such that the following holds for every $0<r<\epsilon$ and every 
	$J$-holomorphic map $u \colon S\to B(0;\epsilon)\subset\R^{2n}$.  

	If $u(\partial S)\subset B(q;r)$ and $\int_{\partial S} u^{\ast}(xdy) \le \frac12 r^{2}$ then $u(S)\subset B(q;Cr)$.
	In particular, there is some constant $c \in \R^{2n}$ such that $\|u - c\|_{C^{0}} = O(r)$. 
\end{lemma}
\begin{proof}
	There exists $\epsilon>0$ such that the area of any such $u$ is bounded by $r^{2}$ by Lemma \ref{l:Sikorav}. The monotonicity lemma for holomorphic curves says that there exists a constants $K>0$ and $\delta>0$ such that if the curve $u$ passes through a point $p\in L$ then the area of intersection of the $\delta$-ball around $p$ and $u$ has area at least $K\delta^{2}$. Now if the image of $u$ contains an arc connecting $\partial B(q;r)$ to $\partial B(q;Cr)$ then this arc contain centers of at least $\frac{(C-1)r}{10\delta}$ disjoint such balls and hence 
	the area of $u$ is at least $K(C-1)\frac{\delta}{10} r$. Taking $C>\frac{10}{K\delta}+1$ proves the lemma. 	
\end{proof}

\begin{lemma} \label{lemma: exponential decay soft}   
	Consider a sequence $(u_{\alpha}, S_\alpha)$ which Gromov converges
	to some $(u, S)$. Let $S^-$ be a connected component of the ghost locus of $u$.  
	Assume the restriction of $u_\alpha$ to ${S_\alpha^-}$ is $J$-holomorphic.  Then: 

	\begin{enumerate}
	\item $\mathrm{Area}(u_{\alpha}(S_\alpha^-)) = \mathcal{O}( \sum \|\, u_\alpha |_{[-1, 0] \times I} \, \|_{C^1}^2)$
		\item \label{controlled ghost image} 
		The image $u(S_{\alpha}^-)$ is contained in a ball of radius $\mathcal{O}( \sum \|\, u_\alpha |_{[-1, 0] \times I} \, \|_{C^1})$ 
		around $c_\alpha$. 
		\item \label{expdecay norms} For any $k>0$, $\|u|_{S_{\alpha}^-}\|_{k}=\mathcal{O}(\sum \|\, u_\alpha |_{[-1, 0] \times I} \, \|_{C^1})$. 
	\end{enumerate} 
	Here $\sum$ means to sum over the different nodes in $S^- \cap S^+$. 
\end{lemma} 
\begin{proof}  
	We apply Lemma \ref{l:Sikorav} to $u_\alpha$.  We may bound independently	
	 the $u_\alpha^* x$ and $u_\alpha^* dy$ 
	 terms, each by $O(\nu)$. By $J$-holomorphicity on 
	 $S_{\alpha}^-$, the $\bar \partial_J$-term vanishes, and the desired estimate of the area follows. 
	
	We obtain the bound on the image from the area estimate using monotonicity.  
	More precisely, since the area is going to zero, monotonicity implies
	that $u_{\alpha}(S_{\alpha}^-)$ is contained in an $\epsilon$-ball for all sufficiently large $\alpha$. 
	Now the boundary of $u_{\alpha}(S_\alpha^-)$ consists of $u_{\alpha}|_{0 \times I}$ 
	along which the 1-form $ydx$ vanishes. 
	The estimate on diameter then follows from Lemma \ref{lem: monotonicity} (the proof of which contains a more
	serious use of monotonicity). 
	
	Estimates on higher derivatives follow by elliptic bootstrapping: use a cut-off function $\beta$ supported near the boundary of $S_{\alpha}^-$. Then
	$\|\beta u\|_{k, \delta}\le C \|(\bar\partial_{J}\beta) u\|_{k-1, \delta}$, where we take the Sobolev norm in a surface obtained by adding a half cylinder $[0,\infty)\times I$ to $S_\alpha^-$ and extend the function by $0$, and we inductively control all derivatives of $u_{\alpha}|_{S_\alpha^-}$.  
\end{proof}

We also note the following.

\begin{corollary} \label{J hol neck decay} 
In the setting of Lemma \ref{lemma: exponential decay soft}, assume that all $u_\alpha$ and $u$ are $J$-holomorphic.
Then the nodal degenerations of cylinders witnessing Gromov convergence may be chosen to have 
exponential neck decay (i.e., satisfy \eqref{neck decay}).  
\end{corollary} 
\begin{proof}
Fix a disk $D$ around the point in the domain of $u$ where the ghost bubble is attached. Let $\{\rho_{j}\}_{j=0}^{\infty}$ be an increasing sequence of numbers with $\rho_{j}\to\infty$. If $\rho_{0}$ is sufficiently large then by Gromov convergence there exists for any $j$ an $\alpha_{j}$ such that for any $\alpha<\alpha_{j}$ there is a region $[\rho_{\alpha}-(\rho_{j}-\rho_{0})-s_{0},\rho_{\alpha}-s_{0}]\times I\subset [-\rho_{\alpha},\rho_{\alpha}]\times I$ such that for $(s,t)$ in this region
\begin{equation}\label{eq: close to u_+} 
\|u_{\alpha}(s,t)-u_{+}(s,t)\|_{C^{1}}\le e^{-4\pi\rho_{j}}.
\end{equation} 
Define $u_{j}\colon [-(\rho_{j}-\rho_{0}),\rho_{j}-\rho_{0}]\times I\to\R^{2n}$ as the restriction of $u_{\alpha_{j}}$ to the region above, suitably translating the coordinates. Then by Lemma \ref{lemma: exponential decay soft}, if $c_{+}$ denotes the asymoptotic constant of $u_{+}$, then
\begin{equation}\label{eq: small area} 
\|(u_{j}-c_{+})|_{[-\rho_{j},0]\times I}\|_{3}=\mathcal{O}(e^{-2\pi\rho_{j}}).
\end{equation}
We find
\begin{align*} 
\|u_{j}-c_{+}\|_{3,\delta}^{\;\wedge} \ &\le  \ e^{\delta\rho_{j}}\|(u_{j}-c_{+})|_{[-\rho_{j},0]\times I}\|_{3} \ + \ \|u_{j}-u_{+}\|_{3,\delta}^{\;\wedge} \ +\ \|(u_{+}-c_{+})|\|_{3,\delta}^{\;\wedge}\\
\ &= \ \mathcal{O}(e^{-(2\pi-\delta)\rho_{j}}) \ + \ \mathcal{O}(e^{-(4\pi-\delta)\rho_{j}}) \ + \ \mathcal{O}(1)
\ = \ \mathcal{O}(1)
\end{align*}       
and we conclude that the sequence $u_{j}\colon[-\rho_{j},\rho_{j}]\times I\to\R^{2n}$ satisfies \eqref{neck decay}. 
\end{proof}

\subsection{Ghost bubble censorship assuming reasonable necks}

\begin{lemma} \label{subleading estimate}
Preserving the hypotheses and notation of Lemma \ref{lemma: exponential decay soft} and assuming in addition
that the sequence has reasonable necks, 
$$
\left\|(u_{\alpha}(s,t)  - c_{\alpha, 0}(0) - c_{\alpha, 1}(0)e^{2\pi(s+it)})|_{[-1,0]\times I}\right\|_k = {\mbox{\tiny$\mathcal{O}$}}(e^{-2\pi\rho_{\alpha}})
$$ 
and the image of $u_\alpha$ is contained in a ball of radius $\cO(e^{2 \pi \rho_\alpha})$. 
\end{lemma}
\begin{proof} 
Consider \eqref{reasonable necks first fourier}.  
The term $\|\bar\partial_{J} u_\alpha|_{[-\rho_{\alpha},0]\times I}\|_{2} = 0$,
as the domain is contained in $S_\alpha^-$ where we assumed $u_\alpha$ to be $J$-holomorphic. 
We have $\|(u_\alpha(s,t)-c_{\alpha,0}(0))|_{[-\rho_{\alpha},-\rho_{\alpha}+1]\times I}\|_{1} = {\mbox{\tiny$\mathcal{O}$}}(1)$ 
as this is a compact subset of the region Gromov converging to a constant map.   We are left
with the displayed formula, from which it follows that $\|(u_\alpha- c_{\alpha,0}(0))|_{[-1, 0] \times I} \, \|_{C^1} = \cO (|c_{\alpha, 1}(0)|)$. By
\eqref{reasonable necks derivative} and \eqref{reasonable necks slow fourier}, 
we have  $|c_{\alpha, 1}(0)| =  \cO(e^{-2 \pi \rho_\alpha})$.  Apply
Lemma \ref{lemma: exponential decay soft} \eqref{controlled ghost image} with this improved estimate. 
\end{proof}

\begin{theorem}\label{thm:compactness}
	Consider a sequence of bare maps
	$u_\alpha\colon (S_\alpha,\partial S_{\alpha}) \to (X,L)$ which Gromov converges to a stable map $u\colon (S,\partial S) \to (X,L)$. 
    Let $u_\mathrm{bare} \colon S_\mathrm{bare} \to X$ be the restriction of $u$ to the components of positive symplectic area, 
	and let $S^-$ be a (nonempty) connected component of $\overline{S \setminus S_\mathrm{bare}}$. 

	Assume the sequence $u_\alpha$ has reasonable necks at the nodes connecting $S^-$ to $ S_\mathrm{bare}$, and 
	 is {\em $J$-holomorphic near $S^-$}.  
	Then one of the following holds: 
\begin{enumerate}
	\item There is a point $p \in S_\mathrm{bare} \cap S^-$ with $\partial_J  u_\mathrm{bare} (T_p S_\mathrm{bare}) = 0$.  
	\item There are two distinct points $p, q \in S_\mathrm{bare} \cap S^-$ with $\partial_J u_\mathrm{bare} (T_p S_\mathrm{bare}) = \partial_J  u_\mathrm{bare}( T_q S_\mathrm{bare})$.
	\item There are at least three distinct points in $S_\mathrm{bare} \cap S^-$. 
\end{enumerate}
\end{theorem}
\begin{proof}
Let $S^+$ be the subcurve of $S$ consisting of all components other than $S^-$.  Then 
$S_{\mathrm{bare}}$ differs from $S^+$ only by ghost bubbles far away from $S^-$.  In particular $S^+ \cap S^- = S_{\mathrm{bare}} \cap S^-$. 

Suppose $S^-$ is attached at a single point, at which $\partial_J u_\mathrm{bare} \ne 0$. 
Then $S^-_\alpha$ has a single boundary component. 
By 
Lemma \ref{subleading estimate} and
 \eqref{reasonable necks derivative} and \eqref{reasonable necks slow fourier}, 
 the projection of $u_\alpha(\partial S_\alpha^-)$ to the line spanned by 
$\partial_J u_\mathrm{bare}$
 winds once 
for all sufficiently large $\alpha$.  

By Lemma \ref{subleading estimate}, the image of $u_\alpha(S_{\alpha}^-)$ is contained
in a ball of radius $O(e^{-2 \pi \rho_\alpha})$, in particular tending to zero, so the difference between 
the ambient almost complex structure and the constant complex structure at $u(S^-)$ 
becomes vanishingly small. 
Thus the Riemann-Hurwitz formula
can be applied, and implies that $S_{\alpha}^-$ is a disk, hence
that the putative limiting ghost bubble is a sphere attached at a single point.  This contradicts stability. 

Suppose $S^-$ is attached at two points, whose complex linear tangent lines are distinct.  
We use Lemma \ref{subleading estimate} as before.  
Consider the projection of $u_\alpha(S_\alpha^-)$ 
onto one of these tangent lines along a subspace containing the other.  Then one component of $\partial S_\alpha^-$
is sent to a large circle winding once around the origin, and the other is sent to a small one winding around the origin 
in the opposite sense.  By Riemann-Hurwitz, 
$S_\alpha^-$ is an annulus, hence the limiting ghost bubble is a sphere attached at two points, contradicting stability. 
\end{proof}

\section{Admissible degenerations have reasonable necks when $J = J_{\mathrm{std}}$}\label{fourier std}  
Here we prove Theorem \ref{admissible reasonable} in the special case $J=J_{\rm std}$.  
To remind ourselves that $J = J_{\rm std}$, throughout this section we name our maps and Fourier
coefficients by different letters: 

\begin{equation}\label{eq:Fourier for holomorhic}
v(s,t) = \sum_{m} h_{m}(s)e^{2\pi m it},\quad \bar\partial v = \sum_{m} b_{m}(s)e^{2\pi m it}=\sum_{m} (h_{m}'(s) - 2\pi m h_{m}(s))e^{2\pi m it}. 
\end{equation}

The arguments below are variations on the following:

\begin{proposition} \label{holomorphic ndc reasonable}
A holomorphic nodal degeneration of cylinders $v_{\alpha}\rightsquigarrow (v_{+},v_{-})$ has reasonable necks and exponential neck decay. 
\end{proposition}
\begin{proof}
 By holomorphicity, the Fourier expansion of $v_{\alpha}\colon[-\rho_{\alpha},\rho_{\alpha}]\times I\to\R^{2n}$ is
\[ 
v_{\alpha}(s,t) \ = \ \sum_{m} h_{\alpha, m} e^{2\pi m (s+it)}, 
\]
for constants $h_{\alpha, m}$.   Now \eqref{reasonable necks slow fourier} is immediate. Also, 
\[ 
h_{\alpha,1}e^{2\pi\cdot \rho_{\alpha}} \ \to \ h_{+,1}e^{2\pi\cdot 0}
\]
by Gromov convergence, so \eqref{reasonable necks derivative} holds.

Next, 
\[
\|(v_{\alpha}(s,t)-h_{0})\|^{2}=
\mathcal{O}\Bigl(\sum_{m<0} |h_{m}|^{2}\Bigr)+\mathcal{O}\Bigl(\sum_{m>0} |h_{m}|^{2}\Bigr),
\]
and
\[
\|(v_{\alpha}(s,t)-h_{0}-h_{1}e^{2\pi(s+it)}|_{[-1,0]\times I}\|^{2}=
\mathcal{O}\Bigl(\sum_{m<0} |h_{m}|^{2}\Bigr)+\mathcal{O}\Bigl(\sum_{m>1} |h_{m}|^{2}\Bigr).
\]
The size of the terms in the right hand sides are controlled by the norm near the ends:
\begin{align*} 
\|v_{\alpha}|_{[-\rho_{\alpha},-\rho_{\alpha}+1]\times I}\|^{2} \ &\ge  \ \frac12 \sum_{m<0} |h_{m}|^{2}e^{4\pi m\rho_{\alpha}} \ \ge \ \frac12 e^{4\pi\rho_{\alpha}}\sum_{m<0} |h_{m}|^{2},\\
\|v_{\alpha}|_{[\rho_{\alpha}-1,\rho_{\alpha}]\times I}\|^{2} \ &\ge \ \frac12 \sum_{m>0} |h_{m}|^{2}e^{4\pi m\rho_{\alpha}} \ \ge \ \frac12 e^{4\pi\rho_{\alpha}}\sum_{m>0}|h_{m}|^{2},\\
\|v_{\alpha}|_{[\rho_{\alpha}-1,\rho_{\alpha}]\times I}\|^{2} \ &\ge \ \frac12 \sum_{m>1} |h_{m}|^{2}e^{4\pi m\rho_{\alpha}} \ \ge \ \frac12 e^{8\pi\rho_{\alpha}}\sum_{m>1}|h_{m}|^{2}.
\end{align*}
The argument for higher derivatives is identical and we find that \eqref{neck decay} and \eqref{reasonable necks first fourier} hold.
\end{proof}

We now turn to the more general setting of admissible degenerations when $J = J_{\mathrm{std}}$.   

\begin{lemma}  \label{lem: limiting cutoff decay standard}
	If a sequence $v_\alpha$ satisfies the cut-off decay condition  \eqref{cut-off decay}, then 
	it satisfies \eqref{reasonable necks slow fourier}. 	
\end{lemma} 
\begin{proof}
	If $\bar\partial v_{\alpha}=\sum_{k} b_{\alpha, k}(s)e^{2\pi k i t}$ then solving the $\bar\pa$-equation for the Fourier coefficient $h_{\alpha, 1}(s)$ of $v_\alpha$ we find
	\[ 
	h_{\alpha, 1}(s) = e^{2\pi s} \left(h_{\alpha, 1}(0) + \int_{0}^{s} e^{-2\pi \sigma}b_{\alpha, 1}(\sigma)d\sigma\right) = h_{\alpha, 1}(0)e^{2\pi s} + E_{\alpha}(s). 
	\]
	Here $|E_{\alpha}(s)|\le \int_{0}^{1}\int_{I}|\bar\partial v|dsdt$, and the cut-off decay condition 
	and the Cauchy-Schwartz inequality imply
	\[ 
	e^{2\pi\rho_{\alpha}}\int_{0}^{1}\int_{I}|\bar\partial_{J}v_{\alpha}| \,dt ds \le \|\bar\partial v_{\alpha}|_{[-1,1]\times I}\|_{0,2\pi}\cdot \left(\int_{-1}^{1}\int_{I} 1 \,dsdt\right)^{\frac12}\to 0,
	\]
	and thus $e^{2\pi\rho_{\alpha}} E_{\alpha}(s)\to 0$ as $\alpha\to\infty$.
\end{proof}

We turn to study  \eqref{reasonable necks derivative}.  We will need the following estimate:  

\begin{lemma}\label{lemma: halpha1 as function of s}
Suppose given a sequence of maps $v_\alpha$ and vectors $\xi_{\alpha}$ satisfying
\eqref{limit compatible} for $J = J_{\mathrm{std}}$. 
 Then for all sufficiently large $s_{0}$, we have
\[ 
h_{1,\alpha}(\rho_{\alpha}-s_{0})= h_{1,\alpha}(1)e^{2\pi(\rho_{\alpha}-s_{0})}+\mathcal{O}(e^{-(2\pi+\delta)s_{0}}).
\]
\end{lemma}

\begin{proof}
Let
\[ 
\bar\partial v_{\alpha}=\sum_{m} b_{m}(s)e^{2\pi m it},
\]
since $\xi_{\alpha}\otimes d\bar z(\partial_{s})= \xi_{\alpha}e^{2\pi s} e^{-2\pi it}$ and since Fourier modes are mutually orthogonal we find
\[ 
\|b_{1}(s)\|_{2\pi+\delta}^{2}\le \sum_{m\ne -1} \|b_{m}(s)\|_{2\pi+\delta}^{2}\le \|\bar\partial v_{\alpha}-\xi_{\alpha}\otimes d\bar z\|_{2\pi+\delta}.
\] 
Then, 
\begin{align*}
h_{1,\alpha}(s+1) &= e^{2\pi s}\left(\int_{1}^{s} e^{-2\pi \sigma} b_{1}(\sigma)\,d\sigma + h_{1,\alpha}(1)\right)\\
&= e^{2\pi s}\left(\int_{1}^{s} e^{-2\pi\sigma}e^{-(2\pi+\delta)(\rho_{\alpha}-\sigma)} e^{(2\pi+\delta)(\rho_{\alpha}-\sigma)} b_{1}(\sigma)\,d\sigma\right) + e^{2\pi s}h_{1,\alpha}(1)\\
&= e^{2\pi(s-\rho_{\alpha})}\left(\int_{1}^{s} e^{-\delta(\rho_{\alpha}-\sigma)} e^{(2\pi+\delta)(\rho_{\alpha}-\sigma)} b_{1}(\sigma)\,d\sigma\right) + e^{2\pi s}h_{1,\alpha}(1)\\
&\le e^{-2\pi(\rho_{\alpha}-s)}\|b_{1}\|_{2\pi+\delta} \frac{1}{\sqrt{\delta}}e^{-\delta(\rho_{\alpha}-s)} + e^{2\pi s}h_{1,\alpha}(1)= e^{2\pi s}h_{1,\alpha}(1) +\mathcal{O}(e^{-(2\pi+\delta)(\rho_{\alpha}-s)}).
\end{align*}
The lemma follows.
\end{proof}

\begin{proposition}\label{prp: dbarcompatible controls neck minus middle std}
	Assume $v_\alpha \rightsquigarrow (v_+, v_-)$ is admissible for $J=J_{\rm std}$. 
	Then, with notation as in Equation \eqref{eq: fourier}, 
	\begin{align}
		\label{plus one fourier limit std}
		\partial v^{\mathrm{disk}}_+ (0) &= \lim_{\rho_\alpha \to \infty} e^{2 \pi (\rho_\alpha-1)} \cdot h_{\alpha, 1}(1)  	 \\
		\label{minus one fourier limit std}
		\partial v^{\mathrm{disk}}_- (0) &= \lim_{\rho_\alpha \to \infty} e^{2 \pi (\rho_\alpha-1)} \cdot h_{\alpha, -1}(-1)  
	\end{align}
	That is, \eqref{reasonable necks derivative} holds. 
\end{proposition}
\begin{proof}
	Let $\bar\partial v_{+}^{\rm disk}(0)=\tau_{+}$ and let $s>0$ be fixed. Recall that $v_{+}^{\rm disk}$ is assumed to be $C^{1, \gamma}$.  
	By Lemma \ref{lemma: s limiting fourier coefficient is derivative}, if $s>0$ is sufficiently large then
	\begin{equation}\label{eq: from Taylor expansion std} 
		h_{+,1}(-s)=\left(\tau_{+}+\mathcal{O}(e^{-\gamma s})\right)e^{-2\pi s}. 
	\end{equation}
	On $[0,-s]\times I$, $v_{\alpha}$ converges to $v_{+}$ in $C^{1}$. This implies that for all sufficiently large $\alpha$,
	\begin{equation}\label{eq: from Gromov convergence std} 
		\|h_{\alpha, 1}(\rho_{\alpha}-s)-h_{+, 1}(-s)\|_{C^{1}}\to 0, \quad\text{ as }\alpha\to\infty
	\end{equation}
	The assumption \eqref{limit compatible} verifies the hypothesis of Lemma \ref{lemma: halpha1 as function of s}, so we find: 
	\begin{equation}\label{eq: from integration std} 
		h_{1,\alpha}(\rho_{\alpha}-s)= h_{1,\alpha}(1)e^{2\pi((\rho_{\alpha}-1)-s)}+\mathcal{O}(e^{-2\pi(\rho_{\alpha}-s)-\delta(\rho_{\alpha}-s)}).
	\end{equation}
We find
\begin{align*} 
|\tau_{+}-e^{2\pi(\rho_{\alpha}-1)}h_{\alpha,1}(1)| &\le
|\tau_{+}-e^{2\pi s}h_{\alpha}(-s)|+e^{2\pi s}|h_{\alpha}(-s)-h_{\alpha, 1}(\rho_{\alpha}-s)|\\
&+|e^{2\pi s}h_{\alpha, 1}(\rho_{\alpha}-s)-h_{1,\alpha}(1)e^{2\pi((\rho_{\alpha}-1)}|.
\end{align*}
Given $\eta>0$ take $s$ such that $\mathcal{O}(e^{-\gamma s})<\frac{\eta}{3}$ and $\mathcal{O}(e^{-2\pi(\rho_{\alpha})-\delta(\rho_{\alpha})})<\frac{\eta}{3}$, then, for all sufficiently large $\alpha$, $e^{2\pi s}|h_{\alpha}(-s)-h_{\alpha, 1}(\rho_{\alpha}-s)|<\frac{\eta}{3}$ by \eqref{eq: from Gromov convergence std}, and consequently $|\tau_{+}-e^{2\pi(\rho_{\alpha}-1)}h_{\alpha,1}(1)|<\eta$. The proposition follows.  
\end{proof}

Note in Lemma \ref{lem: limiting cutoff decay standard} and Lemma \ref{lemma: halpha1 as function of s}, 
we do not assume that the $v_\alpha$ converge as $\alpha \to \infty$.  
We turn to \eqref{reasonable necks first fourier}.   First we show:

\begin{lemma}\label{lemma: c1 dominates}
	Fix some $k \in \N$ and $0<\delta<2\pi$; assume $\rho > 2$.  
	Consider a map $v \colon [-\rho, \rho] \times I \to \R^{2n}$. 
	We use the notation in \eqref{eq:Fourier for holomorhic} for Fourier coefficients and assume $h_0(0) = 0$.  
	\begin{align*} 
		&\|(v(s,t)-h_1(0)e^{2\pi(s+it)})|_{[-1,0]\times I}\|_k \\
		 &= \mathcal{O}\Bigl(e^{-2\pi\rho}\,\|v|_{[-\rho,-\rho+1]\times I}\|_{1} +e^{-4\pi\rho}\|v|_{[\rho-1,\rho]\times I}\|_{1} \\
		&\qquad\quad+ e^{-(2\pi+\delta)\rho}\|\bar\partial v-\sum_{m\le 1} b_{m}(s)e^{2\pi m it}\|_{0, 2\pi+\delta}^{\;\wedge} 
		+ \|\bar\partial v|_{[-\rho_{\alpha},0]\times I}\|_{k-1}
		\Bigr) \\
		&= \mathcal{O}\Bigl(e^{-2\pi\rho}\,\|v|_{[-\rho,-\rho+1]\times I}\|_{1} +e^{-4\pi\rho}\|v|_{[\rho-1,\rho]\times I}\|_{1}\\
		&\qquad 
		+ e^{-(2\pi+\delta)\rho}\|(\bar\partial v-b_{-1}(s)e^{-2\pi it})|_{[1,\rho]\times I}\|_{0, 2\pi+\delta}^{\;\wedge} + \|\bar\partial v|_{[0,1]\times I}\| 
		+\|\bar\partial v|_{[-\rho_{\alpha},0]\times I}\|_{k-1}
		\Bigr) 
	\end{align*}
	The implicit constants in $\cO$ do not depend on $\rho$.	  
\end{lemma}

\begin{proof}
	We have
	\[
	v(s,t)= h_{0}(s) + h_1(s)e^{2\pi it}+\sum_{m>1} h_m(s) e^{2\pi mit} +\sum_{m\le 0} h_m(s) e^{2\pi mit},
	\]
	and for any $m$, 
	$$b_m(s) = h_m'(s) - 2 \pi m h_m(s) = e^{2 \pi  m s} \frac{d}{ds} (e^{-2\pi  m s} h_m(s)),$$ 
	so $$e^{-2\pi  m s} h_m(s) - e^{-2\pi  m s_{0}} h_m(s_{0}) = \int_{s_{0}}^s e^{-2 \pi  m \sigma} b_m(\sigma) d\sigma.$$ 
	
	If $m=0$, since $h_{0}(0)=0$ by assumption, we have for $s\in[-1,0]$,
	\[ 
	|h_{0}(s)|^{2}=\left|\int_{0}^{s} b_{0}(\sigma)\,d\sigma \right|^{2}=\mathcal{O}\left(\|\bar\partial v|_{[-\rho,0]\times I}\|^{2}\right).
	\]
	Consider the contribution of Fourier modes with $m< 0$. Changing variables to take $s\in[0,\rho]$, we have
	\[ 
	h_{m}(-\rho+s)= e^{2\pi m s}\left(\int_{0}^{s}e^{-2\pi m\sigma}b_{m}(-\rho+\sigma)\,d\sigma + h_{m}(-\rho)\right),
	\] 
	and thus for $s\in[\rho-1,\rho]$,
	$$ 
	|h_{m}(-\rho+s)|^{2}=\mathcal{O}\left(e^{4\pi m\rho}|h_{m}(-\rho)|^{2} +\|b_{m}|_{[-\rho,0]\times I}\|^{2}\right)
	$$
	Integrating over $[-1,0]\times I$  gives the contribution from modes with $m\le 0$:
	\begin{equation}\label{eq: squarenorm left}
	\sum_{m\le 0}\|h_{m}|_{[-1,0]}\|^{2}=e^{-4\pi\rho}\mathcal{O}(\|v|_{[-\rho,\rho+1]\times I}\|^{2})+\mathcal{O}(\|\bar\partial v|_{[-\rho,0]\times I}\|^{2})
    \end{equation}
	For Fourier modes with $m>0$, changing variables to take $s\in [0,2\rho]$, we have 
	\[
	h_{m}(\rho-s) = -e^{-2\pi m s}\left(\int_{0}^{s} e^{2\pi m \sigma} b_{m}(\rho-\sigma)\, d\sigma + h_{m}(\rho)\right).
	\]
	For $s\in [\rho-1,\rho+1]$ we write this as
	\begin{align*}
		h_{m}(\rho-s)  = &
		-e^{-2\pi m s}\left(\int_{0}^{\rho-1} e^{(2\pi (m-1)-\delta)\sigma}\cdot e^{(2\pi+\delta)\sigma} b_{m}(\rho-\sigma)\, d\sigma\right)\\ 
		& -e^{-2\pi m s}\left(
		\int_{\rho-1}^{s} e^{2\pi m \sigma} b_{m}(\rho-\sigma)\, d\sigma\right) 
		-e^{-2\pi m s} h_{m}(\rho).
	\end{align*}
	Thus, by Cauchy-Schwartz,
	\begin{align*}
		|h_{m}(\rho-s)|^{2} = & \mathcal{O}\Bigr(
		\;e^{-4\pi m s}\Bigl(\int_{0}^{\rho-1} e^{(4\pi (m-1)-2\delta)\sigma}d\sigma\Bigr)\Bigl(\int_{0}^{\rho-1} e^{(4\pi+2\delta)\sigma} |b_{m}(\rho-\sigma)|^{2}\, d\sigma\Bigr)\\ 
		& \qquad\qquad+\Bigl(\int_{\rho-1}^{s} |b_{m}(\rho-\sigma)|^{2}\, d\sigma\Bigr) 
		+e^{-4\pi m s} |h_{m}(\rho)|^{2}\;\Bigr)\\
		&=\mathcal{O}\Bigl(\; 
		e^{-4\pi ms +(4\pi(m-1)-2\delta)\rho} 
		\|e^{(2\pi+\delta)|\rho-s|}b_{m}(s)\|^{2}\\ 
		&\qquad\qquad+\|b_{m}(s)|_{[0,1]}\|^{2}
		+ e^{-4\pi ms}\|h_{m}(s)|_{[\rho-1,\rho]\times I}|\|_{1}^{2}\;\Bigr).
	\end{align*}
	Thus the contribution to $\|(v -h_1(0)e^{2\pi((s-\rho)+it)})|_{[-1,0]\times I} \|^2$ from modes with $m>1$ is:
	\begin{align*}
		\sum_{m>1} |h_{m}(0)|^{2}\frac{1}{4\pi m}\left(1-e^{-4\pi m}\right) &=\mathcal{O}\Bigl( e^{-(4\pi+2\delta)\rho}\sum_{m>1} \left( \|b_{m}|_{[1,\rho]\times I}\|^{\;\wedge}_{0, 2\pi+\delta} \right)^2\\
		&\qquad\qquad+ \|\bar\partial v|_{[0,1] \times I}\|^{2}+e^{-8\pi\rho}\|v|_{[\rho-1,\rho]\times I}|\|_{1}^{2}\Bigr).
	\end{align*}
	The $L^{2}$-version ($k=0$) of the lemma follows. 
	
	In order to estimate higher derivatives we use the elliptic estimate for $\bar\partial$. More precisely, we include the finite region $[-1,0]\times I$ in the middle of an infinite $\R\times I$ with positive weights, use a cut-off functions $\beta$ with derivative supported in $[-\frac32,-1]\times I\cup[0,\frac12]\times I$ and find
	\[ 
	\| \beta(v - h_{1}(0)e^{2\pi(s+it)}) \|_{k}\le C(\|(\bar\partial \beta) u\|_{k-1}+\|\beta(\bar\partial u)\|_{k-1}).
	\]
	Bootstrapping starting from $k=1$ gives the first stated estimate.  
	The second follows by orthogonality of Fourier modes. 
\end{proof}	
	
\begin{corollary} \label{Jstd rnff} 
Let $v_\alpha \to (v_+, v_-)$ be an admissible nodal degeneration of cylinders, $J = J_{\mathrm{std}}$.  
Then \eqref{reasonable necks first fourier} holds.  
\end{corollary}
\begin{proof}
We apply Lemma \ref{lemma: c1 dominates} with $k=3$ and use Sobolev embedding. 
It remains to explain why  
$$e^{-2\pi\rho}\|v|_{[\rho-1,\rho]\times I}\|_{1} + e^{-\delta \rho}\|(\bar\partial v-b_{-1}(s)e^{-2\pi it})|_{[1,\rho]\times I}\|_{0, 2\pi+\delta}^{\;\wedge} + e^{2 \pi \rho} \|\bar\partial v|_{[0,1]\times I}\| =  {\mbox{\tiny$\mathcal{O}$}}(1)$$
Indeed, each term separately is $ {\mbox{\tiny$\mathcal{O}$}} (1)$, the first by Gromov convergence, the second 
by \eqref{limit compatible}, and the third by \eqref{cut-off decay}. 
\end{proof} 
	
\begin{remark}
A similar estimate of $\|v_{\alpha}-h_{1}(0)e^{2\pi(s+it)}\|_k$ over $[-1,1]\times I$ would require $\|\bar\partial v|_{[0,1]\times I}\|_{k}$ rather than $\|\bar\partial v_{[0,1]\times I}\|$ in the estimate. 
However, in our applications we will have $\bar\partial v=0$ on $[-\rho,0]\times I$ (hence remove the $\|\bar\partial v|_{[-\rho,0]\times I}\|_{2}$-term). 
On the other side of the cylinder, we will have $\bar\partial v\ne 0$.  
The condition \eqref{cut-off decay} says that $0$-norm $\|\bar\partial v\|$ is small on the $e^{-2\pi\rho}$-scale, and to achieve that while 
interpolating between non-zero $\bar\partial v$ and $\bar\partial v=0$, the derivative norm $\|\bar\partial v\|_{1}$ must be quite large on 
the $e^{-2\pi\rho}$-scale.   This is also why, in \eqref{cut-off decay2}, we only ask that  $\|\bar\partial u_\alpha\|_{2}$ is small on the $e^{-\delta \rho_\alpha}$ scale,
for some $\delta$ which in practice will be smaller than $2\pi$. 
\end{remark}

\section{Change of frame}\label{sec: linear}
The Cauchy-Riemann operator acting on a map of a Riemann surface gives a section in the bundle of complex anti-linear maps from the tangent space of the source to the tangent space of the target. In this section we  discuss some general aspects of choice of  frame for such bundles.

\subsection{The change of frame formula} \label{sec: frame change}
Fix a surface $S$, a vector bundle $E$, and a 
connection 
$D\colon \Gamma(E) \to \Gamma(T^*S \otimes_\R E)$, where $\Gamma$ denotes
an appropriate space of sections.  

Now let $J$ be a fiberwise
complex structure on $E$, making it a complex vector bundle, and let 
$j$ be a complex structure on $S$, making it a Riemann surface. 
We take the complex anti-linear part of the connection:  
\begin{equation}
\label{cal connection}
\bar D_J := \tfrac{1}{2}(1 + j \otimes J) \circ D : \Gamma(E) \to \Gamma(T^* S \otimes_\R E).
\end{equation}
If we identify $T^*S \otimes_\R E = T^*S \otimes_\R \C \otimes_\C E$, 
then we see that the image of $\bar D_J$ lands in 
the image of $\tfrac{1}{2}(1 + j \otimes \mathrm{id} \otimes J) = \tfrac{1}{2}(1 + j \otimes i \otimes \mathrm{id})$, 
namely $\Omega^{0,1}_S \otimes_\C E$.  Note that as a map between complex vector spaces,
$$ \bar D_J \colon \Gamma(E) \to \Gamma( \Omega^{0,1}_S \otimes_\C E) $$
is complex linear. 

However, we allow the larger codomain $\Gamma(T^* S \otimes_\R E)$ 
above, in particular because we will want to consider changes of frame which
are not complex linear.  Let $V\colon E \to E$ be a real linear automorphism. 
We have the 
formula: 
\begin{equation} \label{gauge transformation formula}
V^{-1} \circ \bar D_J  \circ V = 
\bar D_{V^{-1} J V} + V^{-1} \circ [\bar D_J V]
\end{equation}
Let us explain what $[\bar D_J V]$ means.  We had $V \in \Gamma(\End(E)) = \Gamma(E \otimes_\R E^*)$.  We apply $\bar D_J$ to the $E$ factor, 
giving  $[\bar D_J V] \in \Gamma(T^*S \otimes_\R E \otimes_\R E^*)$.  

\begin{lemma} \label{linearity} 
For any $V \in \Gamma( \End_\R(E))$, let us write 
$E_{V^{-1} J V}$ for the complex vector bundle with complex structure given by $V^{-1} J V$.  
Then the bundle map 
$$V^{-1} \circ [\bar D_J V] : E_{V^{-1} J V} \to T^*S \otimes_\R E_{V^{-1} J V}$$
has image in $\Omega^{0,1}_S \otimes_\C E_{V^{-1} J V}$, and is complex linear.   
\end{lemma} 
\begin{proof}
As noted above, 
$\bar D_{V^{-1} J V}$ gives a complex linear map 
$\Gamma(E_{V^{-1} J V}) \to \Gamma( \Omega^{0,1}_S \otimes_\C E_{V^{-1} J V}) $.

Consider the operator ${}^V D := V^{-1} \circ D \circ V$.  Then, in the notation of  
\eqref{cal connection},  $V^{-1} \circ \bar D_J  \circ V = \overline{( {}^V D)}_{V^{-1} J V}$,
which thus also gives a complex linear map.  
Hence, the difference $\bar D_{V^{-1} J V} - V^{-1} \circ D \circ V$ is also complex linear. 
\end{proof}

\subsection{Transformation estimates} \label{sec: estimates}  
In this section we study the effect of gauge transformations on various properties of sections. 
Here, and for the remainder of the article, we put ourselves in the following local situation.  
The domain $S$ is a cylinder or strip, $S = [-\rho, \rho] \times I$ and we fix
a trivialization of $\Omega^{0,1}_S$. 
We take the vector bundle to be trivial, $E = \R^{2n}$. 
This carries some complex
structure $J$. We may also consider the standard standard complex structure $J_{\rm std}$.  

For the present subsection, we 
assume given $V \in \Gamma(S, \End_\R(\R^{2n}))$ 
such that 
\begin{equation} \label{miracle} 
	V^{-1} \circ \bar \partial_J  \circ V = \bar \partial
\end{equation}
In terms of the discussion in Section \ref{sec: frame change}, 
we have taken $D = d$, assumed $V^{-1} J V = J_{\mathrm{std}}$, 
and $\bar \partial_J V = 0$. (We will describe how to construct such $V$ 
in Section \ref{Sec:Carlemannearnodes}, below.)

\begin{lemma}
\label{lem: JvsStandard}  
Fix a finite interval $K \subset \R$.  Assume $K \subset (-\rho, \rho)$. 
Then for any $a \in \Gamma((-\rho, \rho) \times I, \C^n)$, 

\begin{equation}\label{eq:Fourier u-v} 
\int_{I} |a (s,t) - V^{-1}(s,t)\, a (s, t)|dt =\mathcal{O}( \|1-V^{-1}\|_{1} \cdot \|a\|_1 ), \qquad s \in K,
\end{equation}  
and
\begin{equation}\label{eq:Fourier du-dv}
\|\bar\partial_J a - \bar\partial (V^{-1} a)\|=\mathcal{O}(
\|1-V^{-1}\|_{1} \cdot \|\bar\partial_J a\|_{1}).
\end{equation}
The norms $\|\cdot \|$ above mean the corresponding norm of the function restricted to $K \times I \subset [-\rho-1, \rho+1] \times I$.
\end{lemma}

\begin{proof}
We have 
\begin{equation*} 
\int_I|a(s,t)- V^{-1}(s,t)\, a(s,t)| dt \le \int_I   |a(s,t)| \cdot   |1- V(s,t)^{-1}|  dt 
\end{equation*}
Recall that for $k>\frac12$, restriction to a hyperplane defines a continuous `trace' map on 
Sobolev spaces, $H^{k}(\R^{d})\to H^{k-\frac12}(\R^{d-1})$.  
In particular,  the $1$-norms of $a$ and $1-V^{-1}$ 
control the $0$-norms of their restrictions to $s$-slices and by Cauchy-Schwarz this controls the $L^1$ norm on the slice.  
By another use of Cauchy-Schwarz, we
conclude that Equation \eqref{eq:Fourier u-v} holds. 

Equation \eqref{eq:Fourier du-dv} follows from 
\[
\|\bar\partial_J a -\bar\partial (V^{-1} a)\| = 
\| (1 - V^{-1})  \bar \partial_J a \| 
=\mathcal{O}(\|1-V^{-1} \|_{1} \| \bar\partial_J a\|_{1}),
\] where we use Cauchy-Schwarz and the fact that $H^{1}$ sits in $L^{4}$.
\end{proof}

Consider a constant section: $c \in \Gamma((-\rho, \rho) \times I, \C^n)$, $dc=0$. 
Generally the transformed $V^{-1} c$ will not be a constant.  Let us denote its Fourier expansion
$$ V^{-1} c = \sum_m h_m(s) e^{2 \pi i  m t} $$
If we introduce the matrix valued function $M_m \in \Gamma( (-\rho, \rho), \End(\R^{2n}))$ 
$$M_m(s) = \int_I e^{- 2 \pi i m t} V^{-1}(s,t)\, dt$$
then
$$h_m(s) = M_m(s) \cdot c$$

\begin{lemma} \label{M0 invertible} 
There exists $\eta > 0 $ such that 
$\|(1 - V^{-1})|_{\{s\} \times I} \|_{L^{1}(I)}  <  \eta$
guarantees the invertibility of $M_0(s)$. \qed
\end{lemma} 

\begin{remark}
Note that the $L_{1}$-norm in Lemma \ref{M0 invertible} is controlled by $\|(1-V^{-1})|_{\{s\}\times I}\|_{C^{0}}$, as well as by $\|(1-V^{-1})|_{K\times I}\|_{1}$ where $K$ is any interval contaning $s$ (by the theorem on trace, i.e., continuity of restriction $H^{k}(\R^{d})\to H^{k-\frac12}(\R^{d-1})$, and Cauchy-Schwarz).
\end{remark}

Using $M_0$, we can describe how to translate sections by constants so that the $V$-conjugate section
has vanishing zeroeth Fourier mode at some given point. 

\begin{definition} \label{def: hat a} 
	For any $a\colon [-\rho,\rho]\times I\to\C^{n}$, 
	we define 	
	$$\hat a(s) := M_0(s)^{-1} \int_I V(s,t)^{-1} a(s, t)\, dt$$ 
	over the locus of $s \in [-\rho, \rho]$ such that $M_0(s)$ is invertible. 
\end{definition}
\begin{lemma} \label{lem: hat a}
	When $\hat a(s)$ is defined, we have
	$$\int_I V(s,t)^{-1} (a(s, t) - \hat a(s)) \,dt = 0$$
	\qed
\end{lemma}

\begin{lemma} \label{estimating constants}
	Assume that $\|(1 - V^{-1})|_{\{s\} \times I} \|_{L^{1}(I)}  <  \epsilon_1$ so that $M_0(s)$ is invertible.  Then for any $c \in \C$, we have 
	$$| \hat a(s) - c | = \mathcal{O}( \|(a(s,t)-c)|_{\{s\} \times I}\|)$$
\end{lemma} 
\begin{proof}
	We have: 
	$$ \hat a(s) = M_0(s)^{-1} \int_{I}V(s,t)^{-1} a(s,t)\,dt =   c + M_0(s)^{-1} \int_{I} V(s,t)^{-1} (a(s,t)-c)\,dt$$ 
	so 
	\begin{align}\label{eq:errorshift} 
	| \hat a(s) - c |  &\le |M_0(s)^{-1}|
	\left(\int_{I} |V^{-1}(s,t)|^{2} \,dt\right)^{\frac12}\left(\int_{I} |a(s,t)-c|^{2} \,dt\right)^{\frac12}\\\notag
	&\le C\left(\int_{I} |a(s,t)-c|^{2} \,dt\right)^{\frac12}.
	\end{align}
\end{proof}

\section{Invertibility of Cauchy-Riemann operators}\label{invCRoperators} 
We recall standard Fredholm properties of the $\bar\partial$-operator on the infinite cylinder and strip. For convenience we state the $L^{2}$-version  separately, even though it is a special case of the more general $L^{p}$ counterpart for $p=2$.

\begin{lemma}\label{l:dbaroncylinder}
	Consider complex valued functions $u\colon \R\times I\to \C$. 
	Fix
	 $\delta_{-},\delta_{+}\notin 2\pi \Z$ and let
	\[ 
	\boldsymbol{\delta}|s|=
	\begin{cases}
		\delta_{-} |s|,&\text{ for }s< 0,\\
		0, &\text{ for }s=0,\\
		\delta_{+} |s|,&\text{ for }s> 0.
	\end{cases}
	\]
	Consider the weighted Sobolev spaces $H^{k}_{(\delta_{-},\delta_{+})}$  with norm 
	\[ 
	\|u\|_{k,(\delta_{-},\delta_{+})}=\left(\int_{\R\times I}\left(\sum_{j=0}^{k} |d^{j}u|^{2}\right) e^{2\boldsymbol{\delta}|s|} dsdt\right)^{\frac12}
	\]  
	and, for $1<p<\infty$,  $W^{k,p}_{(\delta_{-},\delta_{+})}$ with norm
	\[
	\|u\|_{W^{k,p}_{(\delta_{-},\delta_{+})}}=\left(\int_{\R\times I}\left(\sum_{j=0}^{k} |d^{j}u|^{p}\right) e^{p\boldsymbol{\delta}|s|} dsdt\right)^{\frac1p}
	\]
	The operators 
	\[
	\bar\partial\colon H^{k+1}_{(\delta_{-},\delta_{+})}\to H^k_{(\delta_{-},\delta_{+})},\qquad\text{and}\qquad
	\bar\partial\colon W^{k+1,p}_{(\delta_{-},\delta_{+})}\to W^{k,p}_{(\delta_{-},\delta_{+})}	
	\]	
	are Fredholm of index  
	\[ 
	\ind(\bar\partial)=
	\begin{cases}
		\#\{m\in \Z\colon \delta_{-}< 2\pi m< -\delta_{+}\}, &\text{ if }\delta_{-}\le-\delta_{+},\\
		-\#\{m\in \Z\colon -\delta_{-}< 2\pi m< \delta_{+}\}, &\text{ if }-\delta_{-}\le\delta_{+}.
	\end{cases}
	\] 
\end{lemma}

\begin{proof}
	This is well known. For $W^{k,p}_{(\delta_{-},\delta_{+})}$ it is a consequence of \cite[Theorem 1.1]{LockhartMcOwen}. A direct proof for $H^{k}_{(\delta_{-},\delta_{+})}$ can be found in  e.g., \cite[Proposition 6.5]{EES}. We give a brief description of that argument: for small negative weight at one end and small positive at the other the operator is an isomorphism. It is straightforward to verify the elliptic estimate for weights not in $2\pi\Z$ and also how the index changes when the weight passes an integral multiple of $2\pi$.
\end{proof}

\begin{corollary} \label{cor: exun}
	Pick any $\delta$ not an integral multiple of $2 \pi$.  
	Then $\bar \partial\colon H^{k+1}_{(-\delta, \delta)} \to 
	H^{k}_{(-\delta, \delta)}$ is invertible.  
\end{corollary} 
\begin{proof}
	Indeed taking $\delta_+ = \delta = - \delta_-$ a non-integral multiple of $2 \pi$, we see from Lemma \ref{l:dbaroncylinder} 
	that the index is zero. Any solution that satisfies the decay condition implied by the weight at one end must violate the condition at the other, so
	no solution lies in  $H^{k+1}_{(-\delta, \delta)}$,
	the kernel of $\bar \partial$ is zero, and thus $\bar \partial$ must be invertible.   
\end{proof}

Let us recall how Corollary \ref{cor: exun} is used to guarantee existence and uniqueness of solutions to perturbed Cauchy-Riemann equations. 
Invertible operators are open with respect to the operator norm, this means there is some $\epsilon$ such that for any 
$X\colon H^{k+1}_{(-\delta,\delta)}\to H^k_{(-\delta,\delta)}$ 
with operator norm $|X| < \epsilon$, the operator $\bar \partial + X$ remains invertible.  That is, for such $X$ an equation
$(\bar \partial + X) u = v$
will have a unique solution $u \in H^{k+1}_{(-\delta,\delta)}$ whenever $v \in H^k_{(-\delta,\delta)}$.

\section{Carleman similarity for cylinders}\label{Sec:Carlemannearnodes}
Consider on $\R^{2n}$ the standard symplectic form $\omega_{\rm std}$ and some 
compatible almost complex structure $J$.  Let $S$ be a Riemann
surface with complex structure $j$.  We are interested in maps $u\colon S \to \R^{2n}$.  

The Cauchy-Riemann equation for such maps, $\bar \partial_J u = 0$, is nonlinear, 
because $J$ depends on $u$.  However, {\em at a given map $u$}, we may 
consider a related linear equation as follows.  We reinterpret $u$ as 
a section of the trivial $\R^{2n}$-bundle over $S$, and consider the linear
Cauchy-Riemann operator $\bar \partial_{u^* J}$ on this bundle. We may then study the
geometric properties of $u$ using the fact that it solves the linear equation
 $\bar \partial_{u^* J} u = 0$.

The purpose of this section is to show that, over cylinders, the linear operator $\bar \partial_{u^* J}$ can be conjugated
to the standard $\bar \partial$ operator in such a way that certain properties of 
$\bar \partial$ can be transferred to the nonlinear operator $\bar \partial_J$. This circle of ideas has its origins in \cite{carleman} and is often called `Carleman similarity principle', see e.g., \cite[Theorem 2.2]{Floer-Hofer-Salamon} where corresponding results are proved for disk domains.

We conjugate $\bar \partial_{u^*J}$ by combining two results.  
The first concerns how to conjugate $J$ itself.  Here we follow \cite[Remark 2.9]{HWZ-asymptotics}. 
Consider on $\R^{2n}$ an almost complex structure $J$ with $J(0) = J_{\mathrm{std}}$.  
Such a complex structure is compatible with the standard
symplectic form exactly when  $- J(0) J(x)$ is symplectic and positive definite for every $x \in \R^{2n}$.  
In this case the 
(symplectic, positive definite) square root 
\begin{equation} U(x) := (-  J(0) J(x))^{-1/2}\end{equation} 
satisfies
\begin{equation}  \label{eq: conjugate J}
	U(x)^{-1} J(x) U(x) = J(0),  \qquad U(0) = 1. 
\end{equation} 

\begin{lemma}\label{lem: estimate X}   
	Fix $J$ and hence $U$ as above, and
	consider maps $u\colon S \to \R^{2n}$.    We write $\hat U := u^*U$
	and $\hat{J} := u^* J$.  Then 
	$X := \hat{U}^{-1} \circ (\bar \partial_{\hat J} \hat{U}): \C^n \to \Omega^{0,1}_S \otimes_\C \C^n$, where $\C^{n}=(\R^{2n},J_{\rm std})$, is complex linear, 
	and
	\begin{equation} \label{eq:preCarleman} 
		\hat{U}^{-1} \circ \bar \partial_{\hat J} \circ \hat{U} = \bar \partial + X
	\end{equation} 
	Consider $S=[-\rho,\rho]\times I$ and assume $J$ is smooth.  Then there exists $\epsilon > 0$ such if $|u|_{C^0}\le \epsilon$, then for any $\delta\ge 0$
	$$ \|X\|_{k, \delta}^{\;\wedge} = \mathcal{O}(\|du\|_{k, \delta}^{\;\wedge}), \qquad  \|\hat{U}^{\pm 1}-1\|_{k, \delta}^{\;\wedge}=\mathcal{O}(\|u\|_{k, \delta}^{\;\wedge})$$ over any open set in $S$, 
	where the constant in $\mathcal{O}$ is independent of $\rho$.  
\end{lemma}
\begin{proof}  
	We view $u^* T\R^{2n}$ as a trivial $\C^n$ bundle on $U$, with real linear connection given by the de Rham differential. 
	Now Equation \eqref{eq:preCarleman} follows immediately from 
	\eqref{eq: conjugate J} and \eqref{gauge transformation formula}; 
	with the stated properties of $X$ explained in Lemma \ref{linearity}.  
	
	Since $J$ is smooth, for small $\epsilon>0$, we have for $|x|\le\epsilon$, $|U(x)-1|=\mathcal{O}(|x|)$ by Taylor's formula. It follows that $\|1-\hat U^{\pm}\|_{0,\delta}^{\;\wedge}=\mathcal{O}(\|u\|_{0,\delta}^{\;\wedge})$. The degree $k$ derivative is a sum of terms of the form $d^{(k_{0})}U^{\pm}\cdot (d^{(k_{1})}u)^{\gamma_{1}}\dots (d^{(k_{m})}u)^{\gamma_{m}}$, where $\gamma_{1}+\dots+\gamma_{m}=k_{0}$ and $k=(\sum_{j=0}^{m}k_{j})-m$. The desired bound now follows by Cauchy-Schwarz since $H^{1}$ sits in $L^{p}$ for each finite $p$ in dimension $2$ by the limiting case of the Sobolev inequality.

	For $X$, we have by boundedness of $J$ and $U$ and its first derivatives
	\[
	|X| = |\hat U\cdot \bar\partial_{J} \hat U| = |\hat U \cdot (d\hat U + J d\hat U i)|
	    =\mathcal{O}(|du|).
	\]
	It follows that $\|X\|_{0,\delta}^{\;\wedge}=\mathcal{O}(\|du\|_{0,\delta}^{\;\wedge})$. Higher derivatives are estimated analogously, invoking the argument above for higher derivatives of $\hat U$.   
\end{proof}

The second ingredient is that we may conjugate $\bar \partial + X$ to $\bar \partial$. As above, let $\C^{n}=(\R^{2n}, J_{\rm std})$.

\begin{proposition} \label{prop: carleman} 
	Fix $k > 1$ and $n \ge 1$.  Let $S = [-\rho-1,\rho +1 ]\times I$, with its standard complex structure.   
	There exists $\epsilon = \epsilon(k, n) > 0$ such that 
	given any  $ X \in H^k(S, \End_\C(\C^n))$
	 with $\|X\|_k < \epsilon$,
	there exists  
	$P \in \Gamma(S, \End_\C(\C^n))$ such that 
	\begin{enumerate}
		\item \label{carleman Q vanishing} $Q := (\mathbf{1} + P)^{-1} (\bar \partial + X) (\mathbf{1} + P)  - \bar \partial$ vanishes on $[-\rho,\rho]\times I,$   
		\item \label{carleman P Q sobolev estimate} 
		$\|P\|_{k, \delta}^{\;\wedge}, \|Q\|_{k, \delta}^{\;\wedge} = \cO( \|X\|_{k-1, \delta}^{\;\wedge})$,  
		\item\label{carleman P sobolev estimate half neck}
		$\|P|_{[-\rho,-1]\times I}\|_{k, \delta}^{\;\wedge}=\mathcal{O}(\|du|_{[-\rho,0]\times I}\|_{k, \delta}^{\;\wedge}+\|P|_{[-1,0]\times I}\|_{k, \delta}^{\;\wedge})$, 
	\end{enumerate}
where the implicit constants depend only on $n$, $k$, $\delta$, and $\epsilon$.
\end{proposition} 
\begin{proof}
	Using  \eqref{gauge transformation formula} and complex linearity of $P$, we simplify the expression for
	$Q$ to 
	\begin{equation}
	\label{eq: Q}
	Q = (1 + P)^{-1} (\bar \partial P +  XP + X)
	\end{equation} 
	Imagine for a moment
	we could find $P$ satisfying
	\begin{equation}\label{eq:conjugation} 
		\bar\partial P + X P = -X. 
	\end{equation}
	Then we would have $Q = 0$.  
	
	One way to solve Equation (\ref{eq:conjugation}) would be
	to show the operator $\bar \partial + X$ is invertible.  
	Corollary \ref{cor: exun} would provide a tool for this, were it the case that $X$ was defined
	over all of $\R \times I$.  Thus we are led to cut off and extend $X$.

	Consider the infinite cylinder or strip $\R\times I$ and $(-\rho,\rho)\times I\subset \R\times I$. Let $0<\delta<2\pi$. 
	Fix a cut off function $\beta$, equal to $1$ in $[-\rho,\rho]\times I$ and zero outside $[-\rho-\frac{1}{2},\rho+\frac{1}{2}]\times I$.
	Let $X_0 = \beta X$, which we extend by zero to all of $\R \times I$.   Consider the cut-off version of \eqref{eq:conjugation}: 
	\begin{equation}\label{eq:conjugation'} 
		\bar\partial  P_0 +  X_0  P_0 = - X_0. 
	\end{equation}
	
	Suppose given a solution $P_0$ to Equation \eqref{eq:conjugation'}.  Take $P := \beta P_0$.  
	We expand from \eqref{eq: Q}: 
	\begin{eqnarray*}
		(1+P)Q & = & \bar \partial (\beta P_0) + X \beta P_0 + X \\
		&  = &  (\bar \partial \beta)  P_0 + \beta \bar \partial  P_0 +  X_0 P_0 + X  \\
		& = &  (\bar \partial \beta)  P_0 + (\beta - 1) \bar \partial  P_0 + \bar \partial  P_0 +  X_0  P_0 + \beta X + (1-\beta) X \\
		& = & (\bar \partial \beta)  P_0 + (1-\beta)(X - \bar \partial  P_0) 
	\end{eqnarray*} 
	As $\bar \partial \beta$ and $1-\beta$ both vanish on $(-\rho,\rho)\times I$, such $Q$ satisfies
	the desired conclusion (\ref{carleman Q vanishing}) above. 
	
	Now let us show that Equation (\ref{eq:conjugation'}) has a solution. When $I=S^{1}$, Corollary \ref{cor: exun} implies
	$\bar \partial\colon H^{k}_{-\delta} \to 
	H^{k-1}_{-\delta}$ is invertible when restricted to matrix 
	valued functions that vanishes at $(s,t)=(-1,0)$. 

	When $I$ is an interval, the analogous result implies that 
	the same holds with the boundary conditions we have imposed. 
	Thus, there is some $\epsilon'$ such that
	once $|X_0 \cdot|$  (i.e. the norm of the operator $R \mapsto  X_0 R$) 
	is bounded by $\epsilon'$, then $\bar \partial + X_0$ is invertible. 
	
	Recall $H^k(\R^2)$ is a Banach algebra for $k > 1$.  Incorporating the exponential
	weight, we find 
	$\|X_0 R \|_{k-1, -\delta} \le \|X_0 R \|_{k, -\delta} \le \|X_0\|_k \cdot \|R \|_{k, -\delta}$,
	so $|  X_0 \cdot | \le \|X_0\|_{k} = \mathcal{O}(\|X\|_{k})  $.

	Thus we may choose the $\epsilon$ in the hypothesis of the theorem to 
	ensure any given desired bound on $| X_0 \cdot|$, hence invertibility of $\bar \partial +  X_0$. 
	We define 
	\begin{equation}\label{eq:def P}
		P_0 := (\bar \partial + X_0)^{-1} (-X_0).
	\end{equation}

	It remains to estimate $P$ and $Q$. 
	We may estimate the norm of $(\bar \partial +  X_0)^{-1}$ in terms of the norms of $\bar \partial$ 
	(which depends only on $n, k, \delta$) and some contribution controlled by $\epsilon$. 
	Thus
	\[  
	\| P_0\|_{k,-\delta} = \| (\bar\partial  +  X_0)^{-1}(-  X_0) \|_{k, -\delta} = \cO(\| X_0\|_{k-1, -\delta}) 
	\] 
	The same estimate then holds for $P$. For $Q$, we have $\|Q\|_{k;-\delta}=\mathcal{O}(\|P\|_{k;-\delta}+\|X\|_{k;-\delta})$. 
	
	To establish \eqref{carleman P sobolev estimate half neck}, take a cut off function $\phi$ with derivative supported in $[-\frac12,0]\times I$. Again using the operator norm estimate for $\bar\partial +  X_0$, 
	\begin{align*} 
	\|P|_{[-\rho,-1]\times I}\|_{k, -\delta}\le\|\phi P\|_{k, -\delta} &=\mathcal{O}(\|(\bar\partial +  X_0)\phi P\|_{k-1, -\delta})
	=\mathcal{O}(\|\phi X_0\|_{k-1, -\delta}+\|\bar\partial\phi P\|_{k-1, -\delta})\\
	&=
	\mathcal{O}(\|du|_{[-\rho,0]\times I}\|_{k-1, -\delta}+\|P|_{[-1,0]\times I}\|_{k-1, -\delta})
	\end{align*}
\end{proof}
We collect our results: 
\begin{theorem}\label{thm: Carleman}   
	Fix  $k > 1$. 
	For all sufficiently small $\epsilon  > 0$, there is some $\rho_0 = \rho_0(\epsilon)$ such that given any  
	$$u\colon [-\rho-\rho_0-1,\rho + \rho_0 + 1]\times I\to \C^{n}$$ 
	with $\|u\|_{C^0}<\epsilon$ and 
	 $\|du\|_{k, \delta} = \cO(1)$, 
	there exist 
	$$P, Q \colon [-\rho-1,\rho + 1] \times I\to \End_\R(\R^{2n})$$ 
	such that, as operators acting on the trivial $\C^n$ bundle over $[-\rho-1,\rho + 1] \times I$, 
	\begin{equation} \label{eq:Carleman}
		\bar\partial + Q    =  (1+P)^{-1}  \hat U^{-1} \circ \bar \partial_{u^*J} \circ \hat U  (1+P).
	\end{equation} 
	where $\hat U$ is the operator from Lemma \ref{lem: estimate X}. 
	
	Moreover, $P, Q$ may be chosen such that: 
	\begin{enumerate}
		\item $Q = 0$ on $[-\rho,\rho]\times I,$   
		\item $\|P\|_{k, \delta}^{\;\wedge}, \|Q\|_{k, \delta}^{\;\wedge} = \cO( \|du\|_{k-1, \delta}^{\;\wedge})$
		\item
		$\|P|_{[-\rho,-1]\times I}\|_{k, \delta}^{\;\wedge}=\mathcal{O}(\|du|_{[-\rho,0]\times I}\|_{k, \delta}^{\;\wedge}+
		\|P|_{[-1,0]\times I}\|_{k, \delta}^{\;\wedge})$, 
	\end{enumerate} 
	where the implicit constants depend only on $n$, $k$, $\delta$, and $\epsilon$.
\end{theorem}
\begin{proof}
	Note that on a cylinder, $\Omega^{0,1}$ is trivial.  
	We apply Lemma \ref{lem: estimate X} to trade $\hat U^{-1} \bar \partial_J \hat U$ for $\bar \partial + X$
	with $\|X\|_{k, \delta}^{\;\wedge} = \cO(\|du\|_{k, \delta}^{\;\wedge}) = \cO(1)$. 
	We have 
	$\|X|_{[-\rho_{\alpha}-1,\rho_{\alpha}+1]\times I}\|_{k}\le e^{-\delta\rho_{0}}\|X\|_{k;\delta}$. 
	Thus if $\rho_{0}$ is sufficiently large we can apply Proposition \ref{prop: carleman}. 
\end{proof} 

\begin{remark} \label{remark V} 
	In the setting of Theorem \ref{thm: Carleman}, we may take $V = \hat U(1+P)$ and restrict to $[-\rho, \rho]\times I$ to obtain $\bar\partial = V^{-1} \bar \partial_J V$. Then, 
	\begin{equation} \label{eq: v-estimate} \|1-V^{\pm 1}\|_{k, \delta}^{\;\wedge} = \mathcal{O}(\| 1 - \hat U \|_{k,\delta}^{\;\wedge} + \|P\|_{k, \delta}^{\;\wedge}) = \mathcal{O}(\|u\|_{k, \delta}^{\;\wedge}). 
	\end{equation}
We will sometimes also consider maps $u\colon[-\rho,\rho]\times I\to\R^{2n}$ that are centered at $c\in\R^{2n}$ rather than at $0\in\R^{2n}$, in the sense that $\|u-c\|_{k,\delta}^{\;\wedge}$ is small (which means that $\|u\|_{k,\delta}^{\;\wedge}$ is of size $\frac{2}{\sqrt{\delta}}|c|e^{\delta\rho}$). In this case we pick linear coordinates on $\R^{2n}$ so that $J(c)=J_{\rm std}$ and take $U(x)$ to be the square root of $(-J(c)J(x))$ instead. We then obtain $P$ as in Theorem \ref{thm: Carleman} and for $V=\hat U(1+P)$ and the analogue of \eqref{eq: v-estimate} is 
\begin{equation} \label{eq: v-estimate with c} \|1-V^{\pm 1}\|_{k, \delta}^{\;\wedge} = \mathcal{O}(\| 1 - \hat U \|_{k,\delta}^{\;\wedge} + \|P\|_{k, \delta}^{\;\wedge}) = \mathcal{O}(\|u-c\|_{k, \delta}^{\;\wedge}). 
\end{equation} 
\end{remark}

\begin{remark} 
		Geometrically, the linear operator $\bar\partial_{u^{\ast}J}$ can be interpreted as a partial derivative corresponding to freezing the complex structure and varying only the map, as follows. Fix the domain $S$ 
		and consider a configuration space $\mathcal{F}$ of functions $u\colon S\to \R^{2n}$ and the space 
		$\mathcal{J}$ of domain dependent almost complex structures $\hat J\colon S\to\End(\R^{2n})$ on $\R^{2n}$. Thinking of $J$ along $u$ as a domain dependent almost complex structure we write $(u,u^{\ast}J)\in \mathcal{F}\times\mathcal{J}$. There 
		is a natural Cauchy-Riemann operator, which gives a section 
		$\bar\partial_{\mathcal{J}}\colon \mathcal{F}\times\mathcal{J}\to \mathcal{E}$, where $\mathcal{E}$ is the bundle over 
		$\mathcal{F}\times\mathcal{J}$ with fiber over $(u,J)$ given by $(\hat J,j)$-complex anti-linear maps 
		$TS\to u^{\ast}T\R^{2n}$, $\bar\partial_{\mathcal{J}}(u,\hat J)=\bar\partial_{\hat J} u$. Let $\mathcal{U}\subset\mathcal{F}$ be a neighborhood of $u\in\mathcal{F}$ then $\mathcal{V}\subset \mathcal{F}\times\mathcal{J}$ be
		\[ 
		\mathcal{V}=\{(v,K)\colon v\in\mathcal{U}, u^{\ast}J= K\}
		\] 
		Then $\mathcal{U}\subset\mathcal{F}$ gives local coordinates on $\mathcal{V}$. Consider the restriction $\bar\partial_{\mathcal{J}}|_{\mathcal{V}}$. The 
		linear operator $\bar\partial_{u^{\ast} J}$ above is the linearization of the section $\bar\partial_{\mathcal{J}}|_{\mathcal{V}}$ in the local coordinates around $u\in\mathcal{F}$, $\bar\partial_{u^{\ast} J}(v)=\frac{\delta}{\delta v}\bar\partial_{\mathcal{J}}$, where $v\in T_{u}\mathcal{F}$ and we lift it to a tangent vector in $T_{(u,J)}\mathcal{V}$.    
	\end{remark}

\section{Admissible degenerations have reasonable necks} \label{general J fourier} \label{sec: limiting fourier}

\begin{proposition}\label{prop: good parameterization} 
	Fix an almost complex structure $J$ on $\R^{2n}$.  
	Then for any sufficiently small $\epsilon$, there is some $\rho_0 > 0$ such that the following holds. 
	For any admissible nodal degeneration of cylinders 
	$u_\alpha \rightsquigarrow (u_+, u_-)$ in $B(\epsilon)$, 
	parameterized for convenience with domains
	$[-\rho_\alpha -\rho_0 -1, \rho_\alpha + \rho_{0} + 1] \times I$ 
	there are (for sufficiently large $\alpha$) constants $c_\alpha$ and changes of frame $V_\alpha$ 
	over $[-\rho_{\alpha}, \rho_{\alpha}] \times I$, 
	such that 
	$$
	v_\alpha:=V^{-1}_\alpha(u_\alpha-c_\alpha) 
	$$
	satisfies
	$$\bar \partial v_\alpha = V^{-1}_{\alpha} \bar \partial_J u_\alpha \qquad \qquad \mbox{and} \qquad \qquad \int_{I} v_\alpha (0,t)dt=0$$
	along with the estimates
	\begin{align*}
	|c_\alpha - c_{\alpha, 0}(0)| \ &= \ \mathcal{O}(e^{-\delta(\rho_{\alpha}+\rho_{0})}),\\  \|1-V_\alpha^{\pm}\|_{k,\delta}^{\;\wedge} \ &= 
	\ \cO(e^{-\delta\rho_{0}}\|u_\alpha-c_\alpha\|_{k,\delta}^{\;\wedge}).
	\end{align*} 
	Note on the second line that $V_\alpha$ and $u_\alpha$ have different domains; the norms indicated are for the function on its whole domain. 
\end{proposition}

\begin{proof}
	From hypothesis \eqref{neck decay} we see $\| du\|_{2, \delta}^{\;\wedge} = \cO(1)$, 
	so we may apply Theorem \ref{thm: Carleman}.  	
	We define $V_\alpha$ as in Remark \ref{remark V}, ensuring 
	$\bar \partial v_\alpha = V_\alpha^{-1} \bar \partial_J u_\alpha$
	and 
	\begin{align*}
	\|(1-  V^{\pm 1}_{\alpha})|_{[-\rho_{\alpha},\rho_{\alpha}]\times I}\|_{k,\delta}^{\;\wedge} &= \cO(\|(u_\alpha-c_{\alpha,0}(0))|_{[-\rho_{\alpha},\rho_{\alpha}]\times I}|\|^{\; \wedge}_{k, \delta})\\
	&=
	\cO(e^{-\delta\rho_{0}}\|(u_\alpha-c_{\alpha, 0}(0))|_{[-\rho_{\alpha}-\rho_{0},\rho_{\alpha}+\rho_{0}]\times I}|\|^{\; \wedge}_{k, \delta}). 
	\end{align*}
	We take for $k=1$, $\rho_{0}$ large enough that we can deduce invertibility of $M_0(0)$ from Lemma \ref{M0 invertible}. 
	Now we may define $c_\alpha := \hat u_\alpha(0)$.  
	
	Using Lemma \ref{estimating constants}, we find  
	$$
	|c_\alpha - c_{\alpha,0}(0)| = \cO (\|(u_\alpha - c_{\alpha,0}(0))|_{\{0\}\times I}\|).
	$$
	Now, $\|(u_\alpha - c_{\alpha,0}(0))|_{[-1,1]\times I}\|_{1}$ controls $\|(u_\alpha - c_{\alpha,0}(0))|_{\{0\}\times I}\|$ by the theorem on trace (in fact it even controls the $\frac12$-norm) and we find
	\begin{align*}
	|c_\alpha - c_{\alpha,0}(0)| &=\mathcal{O}(\|u_\alpha - c_{\alpha,0}(0)|_{[-1,1]\times I}\|_{1})\\
	&=
	\mathcal{O}(e^{-\delta(\rho_{\alpha}+\rho_{0})}\|e^{\delta|(\rho_{\alpha}+\rho_{0})-s|}(u_\alpha - c_{\alpha,0}(0))|_{[-1,1]\times I}\|_{1})\\
	&=\mathcal{O}(e^{-\delta(\rho_{\alpha}+\rho_{0})}\|(u_\alpha - c_{\alpha,0}(0))|_{[-1,1]\times I}\|_{1,\delta}^{\;\wedge})=\mathcal{O}(e^{-\delta(\rho_{\alpha}+\rho_{0})}).
	\end{align*}
\end{proof}

\begin{proposition} \label{prop: compare limiting fourier} 
	Fix an admissible nodal degeneration of cylinders $u_\alpha \rightsquigarrow (u_+, u_-)$.
	We take from Proposition \ref{prop: good parameterization} conjugate maps $v_\alpha\colon[-\rho_{\alpha},\rho_{\alpha}]\times I\to\R^{2n}$. 
	We write $c_{\alpha, n}$ for the Fourier coefficients of the $u_\alpha$, and $h_{\alpha, n}$ for the Fourier
	coefficients of the $v_{\alpha}$. 
	For $s \in [-\rho_{\alpha}, \rho_{\alpha}]$, we have 
		\begin{equation}\label{comparing fourier}
		|h_{\alpha,1}(s)-c_{\alpha,1}(s)|_{C^{0}}=\mathcal{O}(e^{-2\delta(\rho_{\alpha}-|s|)}),
		\end{equation}
\end{proposition}
\begin{proof} 
	We have
	\begin{align}\label{eq: estimate for comparing Fourier} 
	|h_{\alpha,1}(s)-c_{\alpha,1}(s)|&=\left|\int_{I}(v_{\alpha}(s,t)-u_{\alpha}(s,t))e^{2\pi it}\,dt\right|\\\notag
	&=\left|\int_{I}(v_{\alpha}(s,t)-(u_{\alpha}(s,t)-c_{\alpha}))e^{2\pi it}\,dt\right|\\\notag
	&= \left|\int_{I}(V_{\alpha}^{-1}(s,t)-1)(u_{\alpha}(s,t)-c_{\alpha})e^{2\pi it}\,dt\right|\\\notag
	&\le\left(\int_{I}|1-V_{\alpha}^{-1}(s,t)|^{2}dt\right)^{\frac12}\left(\int_{I}|u_{\alpha}(s,t)-c_{\alpha}|^{2}dt\right)^{\frac12}\\\notag
	&=\mathcal{O}\Bigl(\|(1-V^{-1}_{\alpha})|_{[s-1,s+1]\times I}\|_{1}\Bigr)\mathcal{O}\Bigl(\|(u_{\alpha}-c_{\alpha})|_{[s-1,s+1]\times I}\|_{1}\Bigr)\\\notag
	&=\mathcal{O}\Bigl(e^{-\delta(\rho_{\alpha}-|s|)}\|(1-V^{-1}_{\alpha})\|_{1,\delta}^{\;\wedge}\Bigr)
	\mathcal{O}\Bigl(e^{-\delta(\rho_{\alpha}-|s|)}\|(u_{\alpha}-c_{\alpha})\|_{1,\delta}^{\;\wedge}\Bigr)\\\notag
	&=\mathcal{O}(e^{-2\delta(\rho_{\alpha}-|s|)}),
	\end{align}
where we use the theorem on trace and Proposition \ref{prop: good parameterization} for the estimate on $(1-V_{\alpha}^{-1})$. Integrating \eqref{eq: estimate for comparing Fourier} over $I$ the lemma follows.
	\end{proof}

\begin{lemma} \label{v limit compatible}
If $u_\alpha \rightsquigarrow (u_+, u_-)$ is admissible, then the $v_\alpha$ from Proposition \ref{prop: good parameterization} is $\bar\partial$-compatible, i.e.,
satisfies \eqref{limit compatible} with respect to $J=J_{\rm std}$. 
\end{lemma} 
\begin{proof}
Recall that $v_{\alpha}=V_{\alpha}^{-1}(u_{\alpha}-c_{\alpha})$ and by assumption there exists $\xi_{\alpha}$ such that  \eqref{limit compatible}:
\[ 
\|(\bar\partial_{J}u_{\alpha}-\xi_{\alpha}\otimes d\bar z)|_{[1,\rho_{\alpha}]\times I}\|_{0,2\pi+\delta}=\mathcal{O}(1)
\]
holds. We want to show that 
\[ 
\|(\bar\partial v_{\alpha}-\xi_{\alpha}\otimes d\bar z)|_{[1,\rho_{\alpha}]\times I}\|_{0,2\pi+\delta}=\mathcal{O}(1).
\]
	
We have 
\begin{align}\label{eq: dbar v and dbar u 1}
	\bar\partial v_{\alpha} = V_{\alpha}^{-1}\bar\partial_{J} u_{\alpha}= V_{\alpha}^{-1}(\bar\partial_{J} u_{\alpha}-\xi_{\alpha}\otimes d\bar z)+(V_{\alpha}^{-1}-1)\xi_{\alpha}\otimes d\bar z + \xi_{\alpha}\otimes d\bar z.
\end{align}	
Here 
\begin{equation}\label{eq: dbar v and dbar u 2}
\|V_{\alpha}^{-1}(\bar\partial_{J} u_{\alpha}-\xi_{\alpha}\otimes d\bar z)\|_{0,2\pi+\delta}^{\;\wedge}=\|V_{\alpha}\|_{C^{0}}\,\|\bar\partial_{J} u_{\alpha}-\xi_{\alpha}\otimes d\bar z\|_{0,2\pi+\delta}^{\;\wedge}=\mathcal{O}(1)
\end{equation}
and
\begin{equation}\label{eq: dbar v and dbar u 3}
|(V_{\alpha}^{-1}(s,t)-1)|\cdot |\xi_{\alpha}\otimes d\bar z|=\mathcal{O}(e^{-(\delta+2\pi)(\rho_{\alpha}-s)}),
\end{equation}
by Proposition \ref{prop: good parameterization}. 
This implies that 
\begin{equation}\label{eq: dbar v and dbar u 4}
b_{-1,\alpha}(s)=\int_{I}\bar\partial v_{\alpha}e^{-2\pi i t}\,dt= \xi_{\alpha}e^{-2\pi (\rho_{\alpha}-s)} + \mathcal{O}(e^{-(2\pi+\delta')(\rho_{\alpha}-s)}),
\end{equation}
for any $\delta'<\delta$ and then that
\begin{align*}
\|\bar\partial v_{\alpha}- \xi_{\alpha}\otimes d\bar z\|_{0,2\pi+\delta'} & \ \le 
\ \|\bar\partial v_{\alpha}- b_{\alpha,-1}(s)e^{-2\pi i t}\|_{0,2\pi+\delta'}^{\;\wedge}
+ \ \|b_{\alpha,-1}(s)e^{-2\pi i t}-\xi\otimes dz\|_{0,2\pi+\delta'}^{\;\wedge}\\
& \  = \ \mathcal{O}(1).
\end{align*}
\end{proof} 

\begin{corollary}\label{cor: dbarcompatible controls neck minus middle}
	Assume $u_\alpha \rightsquigarrow (u_+, u_-)$ is admissible. 
	Then, with notation as in Equation \eqref{eq: fourier}, 
	\begin{align}
		\label{plus one fourier limit}
		\partial_J u^{\mathrm{disk}}_+ (0) &= \lim_{\rho_\alpha \to \infty} e^{2 \pi (\rho_\alpha-1)} \cdot c_{\alpha, 1}(1)  	 \\
		\label{minus one fourier limit}
		\partial_J u^{\mathrm{disk}}_- (0) &= \lim_{\rho_\alpha \to \infty} e^{2 \pi (\rho_\alpha-1)} \cdot c_{\alpha, -1}(-1),  
	\end{align}
i.e., $u_\alpha \rightsquigarrow (u_+, u_-)$ satsifies the reasonable neck condition \eqref{reasonable necks derivative}.
\end{corollary}
\begin{proof}
	We show the result for $\partial_{J}u_{+}^{\rm disk}(0)$, $\partial_{J}u_{-}^{\rm disk}(0)$ is directly analogous.
	Let $\bar\partial_J u_{+}^{\rm disk}(0)=\tau_{+}$ and let $s>0$ be fixed. Recall that $u_{+}^{\rm disk}$ is assumed to be $C^{1, \gamma}$.  By
	Lemma \ref{lemma: s limiting fourier coefficient is derivative}, if $s>0$ is sufficiently large then
	\begin{equation}\label{eq: from Taylor expansion} 
	c_{+,1}(-s)=\left(\tau_{+}+\mathcal{O}(e^{-\gamma s})\right)e^{-2\pi s}. 
	\end{equation}
	 On $[0,-s]\times I$, $u_{\alpha}$ converges to $u_{+}$ in $C^{1}$. This implies that for all sufficiently large $\alpha$,
	\begin{equation}\label{eq: from Gromov convergence} 
	\|c_{\alpha, 1}(\rho_{\alpha}-s)-c_{+, 1}(-s)\|_{C^{1}}\to 0, \quad\text{ as }\alpha\to\infty
	\end{equation}
	By Proposition \ref{prop: compare limiting fourier},
	\begin{equation}\label{eq: from (u-c) square} 
	\|h_{\alpha,1}(\rho_{\alpha}-s)-c_{\alpha,1}(\rho_{\alpha}-s)\|_{C^{0}}=\mathcal{O}(e^{-2\pi s} e^{-2(\delta-\pi)s}).
	\end{equation}
	Using Lemma \ref{v limit compatible} to verify the hypothesis of Lemma \ref{lemma: halpha1 as function of s}, we find: 
	\begin{equation}\label{eq: from integration} 
	h_{1,\alpha}(\rho_{\alpha}-s)= h_{1,\alpha}(1)e^{2\pi((\rho_{\alpha}-1)-s)}+\mathcal{O}(e^{-2\pi(\rho_{\alpha}-s)-\delta(\rho_{\alpha}-s)}),
	\end{equation}
	We then have
	\begin{align*}
		|e^{2\pi(\rho_{\alpha}-1)}c_{\alpha,1}(1)-\tau_{+}| \ \le \ & |e^{2\pi(\rho_{\alpha}-1)}c_{\alpha,1}(1)- e^{2\pi(\rho_{\alpha}-1)}h_{\alpha,1}(1)|\\
		+ \ & |e^{2\pi(\rho_{\alpha}-1)}h_{\alpha,1}(1)-e^{2\pi s}h_{\alpha,1}(\rho_{\alpha}-s)| \\ 
		+ \ & |e^{2\pi s}h_{\alpha,1}(\rho_{\alpha}-s) - e^{2\pi s}c_{\alpha,1}(\rho_{\alpha}-s)|\\
		+ \ & |e^{2\pi s}c_{\alpha,1}(\rho_{\alpha}-s) - e^{2\pi s}c_{+,1}(-s)|\\
		+ \ & |e^{2\pi s}c_{+,1}(-s)-\tau_{+}|\\
	\text{[by \eqref{eq: from (u-c) square} for $s=\rho_{\alpha}-1$]} \ = \ & \;\;\; \mathcal{O}(e^{-2(\delta-\pi)\rho_{\alpha}})\\ 
	\text{[by \eqref{eq: from integration}]}\qquad & \ + \  \mathcal{O}(e^{-2\pi\rho_{\alpha}-\delta(\rho_{\alpha}-s)})\\
	\text{[by \eqref{eq: from (u-c) square}]}\qquad & \ + \  \mathcal{O}(e^{-2(\delta-\pi)s})\\
	 &\ + \   e^{2\pi s}|c_{\alpha,1}(\rho_{\alpha}-s) - c_{+,1}(-s)|\\
	\text{[by \eqref{eq: from Taylor expansion}]}\qquad &\ + \  \mathcal{O}(e^{-\gamma s}).
	\end{align*}

Given $\eta>0$ take $s>0$ so that $\mathcal{O}(e^{-2(\delta-\pi)s})<\frac{\eta}{5}$ and $\mathcal{O}(e^{-\gamma s})<\frac{\eta}{5}$, then for all sufficiently large $\alpha$, $\mathcal{O}(e^{-2(\delta-\pi)\rho_{\alpha}})<\frac{\eta}{5}$ and $\mathcal{O}(e^{-2\pi\rho_{\alpha}-\delta(\rho_{\alpha}-s)})<\frac{\eta}{5}$ and by \eqref{eq: from Gromov convergence} $e^{2\pi s}|c_{\alpha,1}(\rho_{\alpha}-s) - c_{+,1}(-s)|<\frac{\eta}{5}$. We conclude that
\[ 
|e^{2\pi(\rho_{\alpha}-1)}c_{\alpha,1}(1)-\tau_{+}|<\eta.
\]
The lemma follows.
\end{proof}

Additionally, the Fourier coefficient corresponding to the complex derivative in the left (right) part of the limit curve does not change much over the middle region.

\begin{corollary}\label{l:c_1 does not change much}
	Assume $u_\alpha \rightsquigarrow (u_+, u_-)$ is admissible. 
	 	Then for $-1\le s\le 1$,  we have $c_{\alpha, 1}(s) = e^{2\pi s}c_{\alpha, 1}(0) + {\mbox{\tiny$\mathcal{O}$}}(e^{-2\pi\rho_{\alpha}})$, i.e., $u_\alpha \rightsquigarrow (u_+, u_-)$ satsifies the reasonable neck condition \eqref{reasonable necks slow fourier}. 
\end{corollary}

\begin{proof} 
	Immediate from Lemma \ref{lem: limiting cutoff decay standard} and Lemma \ref{lem: JvsStandard} and 
	Proposition \ref{prop: compare limiting fourier}.  
\end{proof}

We also have the following from conjugating Lemma \ref{lemma: c1 dominates}.   

\begin{lemma}\label{lemma: c1 dominates near holomorphic ghost}
	Let $1<k\le 3$.  Consider an admissible degeneration $u_\alpha \rightsquigarrow (u_+, u_-)$.
	Let $c_{\bullet, n}$ be the Fourier coefficients of $u_\bullet$.    	
	Then there exists $\alpha_{0}$ such that for all $\alpha>\alpha_{0}$:  
	\begin{align*}
		\left\|(u_{\alpha}(s,t)  - c_{\alpha, 0} - c_{\alpha, 1}(0)e^{2\pi(s+it)})|_{[-1,0]\times I}\right\|_k \ &= \
		{\mbox{\tiny$\mathcal{O}$}}(e^{-2\pi\rho_{\alpha}}) + \mathcal{O}(\|\bar\partial_{J} u_\alpha|_{[-\rho_{\alpha},0]\times I}\|_{k-1})\\
		&+ \
		\mathcal{O}\Bigl(e^{-2\pi\rho_{\alpha}}\|(u-c_{\alpha})|_{[-\rho_{\alpha},-\rho_{\alpha}+1]\times I}\|_{1}\Bigr),
	\end{align*}
i.e., $u_\alpha \rightsquigarrow (u_+, u_-)$ satisfies the reasonable neck condition \eqref{reasonable necks first fourier}. 	
\end{lemma}

\begin{proof} 
	Let $v_{\alpha}=V_{\alpha}^{-1}(u_{\alpha}-c_{\alpha})$.  We have
	\begin{align}\label{eq : c1 dom term I}
	u_{\alpha}-c_{\alpha}-c_{\alpha,1}(0)e^{2\pi(s+it)} & \ = \ (u_{\alpha}-c_{\alpha}) - v_{\alpha}\\ \label{eq : c1 dom term II}
	& \ + \ v_{\alpha} - h_{\alpha,1}(0)e^{2\pi(s+it)} \\ \label{eq : c1 dom term III}
	& \ + \ (h_{\alpha,1}(0) - c_{\alpha,1}(0))e^{2\pi(s+it)}. 
	\end{align}
	We first estimate the right hand term in \eqref{eq : c1 dom term I}. By  Propositions \ref{prop: good parameterization} 
	we have 
	\begin{align}\label{eq : c1 dom term I est}
	\|((u_{\alpha}-c_{\alpha}) - v_{\alpha})|_{[-1,0]\times I}\|_{k} &= \mathcal{O}(\|(1-V_{\alpha}^{-1})|_{[-1,0]\times I}\|_{k}\|(u_{\alpha}-c_{\alpha})|_{[-1,0]\times I}\|_{k})\\\notag
	&= \mathcal{O}(\|(u_{\alpha}-c_{\alpha})|_{[-1,0]\times I}\|^2_{k})
	=\mathcal{O}(e^{-2\delta\rho_{\alpha}}) = {\mbox{\tiny$\mathcal{O}$}}(e^{-2\pi\rho_{\alpha}}) .	
	\end{align}	
	(Recall we assumed $\delta > \pi$.)
	
	By Proposition \ref{prop: compare limiting fourier} we have for \eqref{eq : c1 dom term III}  
	\begin{equation}\label{eq : c1 dom term III est}
	|h_{\alpha,1}(0) - c_{\alpha,1}(0)|=\mathcal{O}(e^{-(2\pi + 2(\delta-\pi))\rho_{\alpha}}) = {\mbox{\tiny$\mathcal{O}$}}(e^{-2\pi\rho_{\alpha}}).
	\end{equation} 
	
	We then consider \eqref{eq : c1 dom term II}, by Lemma \ref{lemma: c1 dominates} (take $0<\delta'<\delta$) 
	\begin{align*} 
		\|(v_{\alpha}(s&,t)-h_{\alpha,1}(0)e^{2\pi(s+it)})|_{[-1,0]\times I}\|_k \\
		& = \mathcal{O}\Bigl(e^{-2\pi\rho}\,\|v_{\alpha}|_{[-\rho_{\alpha},-\rho_{\alpha}+1]\times I}\|_1
		+e^{-4\pi\rho}\|v|_{[\rho-1,\rho]\times I}\|_1 \\
		&+ e^{-(2\pi+\delta')\rho_{\alpha}}\|(\bar\partial v_{\alpha}-b_{\alpha,-1}(s)e^{-2\pi it})|_{[0,\rho_{\alpha}]\times I}\|_{0, 2\pi+\delta'}^{\;\wedge} + \|\bar\partial v_{\alpha}|_{[-1,1]\times I}\| +\|\bar\partial v|_{[-\rho_{\alpha},0]\times I}\|_{k-1} \Bigr) 
	\end{align*}
	Here the first term
	\begin{align*} 
	e^{-2\pi\rho}\,\|v_{\alpha}|_{[-\rho_{\alpha},-\rho_{\alpha}+1]\times I}\|_1 &= 
	e^{-2\pi\rho_{\alpha}}\,\|V_{\alpha}^{-1}(u_{\alpha}-c_{\alpha})|_{[-\rho_{\alpha},-\rho_{\alpha}+1]\times I}\|_1 \\
	&=\mathcal{O}(e^{-2\pi\rho_{\alpha}}\|(u_{\alpha}-c_{\alpha})|_{[-\rho_{\alpha},-\rho_{\alpha}+1]\times I}\|_1 ),
	\end{align*} 
	and the second
	\[ 
	e^{-4\pi\rho_{\alpha}}\|v_{\alpha}|_{[\rho-1,\rho]\times I}\|_1 =e^{-4\pi\rho_{\alpha}}\|V_{\alpha}^{-1}(u_{\alpha}-c_{\alpha})|_{[\rho-1,\rho]\times I}\|_1 =\mathcal{O}(e^{-4\pi\rho_{\alpha}}),
	\]
	since Gromov convergence implies $\|(u_{\alpha}-c_{\alpha})|_{[\rho-1,\rho]\times I}\|_1 = \cO(1)$. 
	For the third term, \eqref{eq: dbar v and dbar u 1} -- \eqref{eq: dbar v and dbar u 4} in the proof of Lemma \ref{v limit compatible} shows  
    that
    \[ 
    \|\bar\partial v_{\alpha}- b_{-1,\alpha}(s)\|_{0,2\pi+\delta'}=\mathcal{O}(1).
    \]
For the fourth term, note that $\bar\partial v_{\alpha}=V_{\alpha}^{-1}\bar\partial_{J}u_{\alpha}$ and 
\[ 
\|\bar\partial v_{\alpha}|_{[-1,1]\times I}\|\le \|V_{\alpha}^{-1}|_{[-1,1]\times I}\|_{C^{0}}\,\|\bar\partial_{J} u_{\alpha}|_{[-1,1]\times I}\|={\mbox{\tiny$\mathcal{O}$}}(e^{-2\pi\rho_{\alpha}}),
\]
by \eqref{cut-off decay}.  Similarly, for the final term, when $k=1$: 
\begin{align*}
\|\bar\partial v_{\alpha}|_{[1, \rho]\times I}\| = \|V_{\alpha}^{-1}\bar\partial_{J}u_{\alpha}\|=\cO(\|V_{\alpha}^{-1}\|_{C^{0}}\cdot\|\bar\partial_J u_{\alpha}|_{[1, \rho]\times I}\| )=\cO(\|\bar\partial_J u_{\alpha}|_{[1, \rho]\times I}\| ),
\end{align*}
since $\|V_{\alpha}^{-1}\|_{C^{0}}=\mathcal{O}(\|V_{\alpha}^{-1}\|_{2})$. When $k=2$, we have also
\begin{align*}
	\|d(\bar\partial v_{\alpha})|_{[1, \rho]\times I}\| &= 
	\mathcal{O}\Bigl(\|(dV_{\alpha}^{-1}\cdot du_{\alpha})\bar\partial_{J}u_{\alpha}\|+\|V_{\alpha}^{-1}\cdot d\bar\partial_{J} u\|\Bigr)\\
	&=\cO(\|dV_{\alpha}^{-1}\cdot d u_{\alpha}\|_{C^{0}}\cdot\|\bar\partial_J u_{\alpha}|_{[1, \rho]\times I}\|
	+\|\bar\partial_{J} u_{\alpha}|_{[1, \rho]\times I}\|_{1}\Bigr)\\
	&=\mathcal{O}\Bigl(\|\bar\partial_J u_{\alpha}|_{[1, \rho]\times I}\|+\|\bar\partial_{J} u_{\alpha}|_{[1, \rho]\times I}\|_{1}\Bigr),
\end{align*}
where we use $\|du_{\alpha}\|_{2}\le \|u-c_{\alpha}\|_{3}$. When $k=3$ we use the Banach algebra property of $H^{2}$, and in all cases we conclude
\[ 
\|\bar\partial v_{\alpha}|_{[-\rho_{\alpha},0]\times I}\|_{k-1}=\mathcal{O}(\|\bar\partial_{J} u_{\alpha}|_{[-\rho_{\alpha},0]\times I}\|_{k-1}).
\]

Then 
\begin{align}\label{eq : c1 dom term II est}
	\|(v_{\alpha}(s,t)-h_{\alpha,1}(0)e^{2\pi(s+it)})|_{[-1,0]\times I}\| &={\mbox{\tiny$\mathcal{O}$}}(e^{-2\pi\rho_{\alpha}}) + \cO(\|\bar\partial_J u_{\alpha}|_{[1, \rho]\times I}\| )
	\\ \notag &+  \mathcal{O}(e^{-2\pi\rho_{\alpha}}\|(u_{\alpha}-c_{\alpha})|_{[-\rho_{\alpha},-\rho_{\alpha}+1]\times I}\|).
\end{align}

Equations \eqref{eq : c1 dom term I est}, \eqref{eq : c1 dom term III est}, and \eqref{eq : c1 dom term II est} implies the lemma with $c_\alpha$ in place of $c_{\alpha, 0}$.  Integrating
$u_{\alpha}$ over $0\times I$ we find $|c_\alpha - c_{\alpha, 0}|$ satisfies the same estimate, so we may replace
$c_{\alpha}$ with $c_{\alpha, 0}$. 
\end{proof}

We collect our results: 

\begin{theorem} \label{admissible reasonable}
An admissible nodal degeneration of cylinders has reasonable necks.
\end{theorem}
\begin{proof}
Lemma \ref{lemma: c1 dominates near holomorphic ghost}
with $k = 3$ implies Condition \eqref{reasonable necks first fourier} by Sobolev embedding.
Corollaries \ref{cor: dbarcompatible controls neck minus middle} and \ref{l:c_1 does not change much} establish
Conditions \eqref{reasonable necks derivative} and \eqref{reasonable necks slow fourier}, respectively.  
\end{proof}

\section{Exponential neck decay}\label{sec:Wkp with weights}
In this section we show that for maps from a cylinder into $\R^{2n}$, the $(k, \delta)$ norm is controlled 
by the norms of the restriction to the ends, plus the anti-holomorphic
derivative.  This in turn allows to conclude exponential decay in the neck region 
from the corresponding property of its $J$-complex anti-linear derivative.

\begin{proposition} \label{prop: constant control}
 Let $0<\delta<2\pi$. Then there exists $\epsilon>0$ and $C>0$, depending only on $J$ and $\delta$, such that 
 for any $u\colon[-\rho-1,\rho+1]\times I\to\R^{2n}$ such that 
 $|u|_{C^{0}}<\epsilon$ and $\|\bar\partial u|_{[-\rho-1,\rho+1]\times I}\|_{1,\delta}^{\;\wedge}<\epsilon$, there exists $c\in\R^{2n}$, with $|c|<\epsilon$ such that
	\[ 
	\|(u-c)|_{[-\rho,\rho]\times I}\|_{3,\delta}^{\;\wedge}  \le C(\|\bar\partial_{J} u\|_{2, \delta}^{\;\wedge} +\|(u-c)|_{[-\rho-1,\rho]\times I}\|_{2}+\|(u-c)|_{[\rho,\rho+1]\times I}\|_{2})
	\]
\end{proposition}
\begin{proof}
We will repeatedly use the elliptic estimate in weighted Sobolev $p$-norms.  We denote: 
\[ 
\|w\|_{W^{k,p}_{\delta}}=\left(\int_{\R\times I}\left(\sum_{k} |d^{(k)}w|^{p}\right)e^{\delta p|s|}dsdt\right)^{\frac{1}{p}}.
\]
We sometimes write $L^{p}_{\delta}$ for $W^{0,p}_{\delta}$.

Recall from \eqref{big in middle} that for $w\colon\R\times I\to\R^{2n}$, we have
\[ 
\|w|_{[-\rho,\rho]\times I}\|_{k,\delta}^{\;\wedge} = e^{\delta\rho} \|w\|_{k, -\delta} = e^{\delta\rho} \|w\|_{W^{k,2}_{-\delta}}.
\] 


Pick coordinates so that $J(0)=J_{\rm std}$, and write
\begin{align*} 
\bar\partial_{J}u &= \tfrac12(du + J(u)\circ du\circ i)\\
&= \tfrac12(du + J_{\rm std}\circ du\circ i) +\tfrac12(J(u)-J_{\rm std})\circ du\circ i\\
&=\bar\partial u + A(s,t)\circ du\circ i.
\end{align*}
Let $\beta_{0}\colon [-\rho-1,\rho+1]\times I\to [0,1]$ be a cut off function such that
\[ 
\beta_{0}(s,t)=
\begin{cases}
0, & (s,t)\in [-\rho-1,-\rho-\frac78]\times I\cup [\rho+\frac78,\rho+1]\times I,\\
1, & (s,t)\in [-\rho-\frac34,\rho+\frac34]\times I.
\end{cases}
\]

Consider the linear operator on functions $w\colon \R\times I\to\R^{2n}$ given by
\[ 
G_0 w=\bar\partial w + (\beta_{0}\cdot A)\circ dw \circ i.
\]
Let us take the domain and codomain as $G_{0}\colon W^{1,p}_{-\delta}\to L^{p}_{-\delta}$. 
By Lemma \ref{l:dbaroncylinder}, $\bar\partial\colon W^{1,p}_{-\delta}\to L^{p}_{-\delta}$ is Fredholm of index $2n$ with kernel given by constant functions. Therefore, $\bar\partial$ is invertible on the $2n$-codimensional intersection
\[
\bigcap_{j=1}^{2n} \mathrm{ker}(l_{j}) \ \subset \ W^{1,p}_{-\delta},
\] 
where $l_{j}\colon W^{1,p}_{-\delta}\to\R$ is defined as 
\[
l_{j}(w)= \int_{\R\times I}(w\cdot e_{j})\, e^{-2\delta |s|}\;dsdt,
\]
where $e_{j}$ is the $j^{\rm th}$ coordinate vector in $\R^{2n}$. The continuity if $l_{j}$ follows from  H\"older's inequality: if $\frac{1}{q}=1-\frac{1}{p}$ then
\begin{align*}
|l_{j}(w)| &\le \int_{\R\times I} |w\cdot e_{j}| e^{-2\delta |s|}\;dsdt \\
&\le \left(\int_{\R\times I} |w|^{p} e^{-p\delta |s|}\;dsdt\right)^{\frac1p}
\left(\int_{\R\times I}  e^{-q\delta |s|}\;dsdt\right)^{\frac1q} = \left(\frac{2}{\delta q}\right)^{\frac1q}\|w\|_{W^{1,p}_{-\delta}}. 
\end{align*}

Since $u$ maps into an $\epsilon$-ball we have $|\beta_{0}\cdot A|_{C^{0}}=\mathcal{O}(\epsilon)$.
It follows that $G$ can be made arbitrarily close to $\bar\partial$ by taking $\epsilon$ is sufficiently small, so the same invertibility property holds for the operator $G_0$.  Therefore, there exists a constant $C_{0}$, such that
\begin{equation}\label{eq:ellipticG} 
	\|w-b_{w}\|_{W^{1,p}_{-\delta}}\le C_{0}\|G_0 w\|_{L^{p}_{-\delta}} 
\end{equation}
where $b_{w}=l_{0}^{-1}\sum_{j=1}^{2n} l_{j}(w)$, $l_{0}=l_{j}(e_{j})$, is the weighted average of $w$
(so that $w-b_{w}$ is in the  $\bigcap{\mathrm{ker}}(l_{j})$). 

We apply \eqref{eq:ellipticG} to  $w =\beta_{0}\cdot u$ (extending the RHS by zero to all of $\R\times I$):
\begin{align}\label{eq:1norm(u-b)} 
	\|\beta_0\cdot u- b_{\beta_0\cdot u}\|_{W^{1,p}_{-\delta}}\le C_{0} \|G_{0}(\beta_{0}\cdot u)\|_{L^{p}_{-\delta}}\le C_{0;1}(\|\beta_{0}\cdot\bar\partial_{J} u\|_{L^{p}_{-\delta}}+ \|d\beta_0 \cdot u\|_{L^{p}_{-\delta}}).
\end{align} 

The first term on the right hand side of \eqref{eq:1norm(u-b)} satisfies 
\[ 
\|\beta_{0}\cdot\bar\partial_{J} u\|_{L^{p}_{-\delta}} \le C \|\bar\partial u|_{[-\rho-1,\rho+1]\times I}\|_{1, -\delta}=\mathcal{O}(\epsilon e^{-\delta\rho}),
\] 
by limiting case of the Sobolev inequality as follows. Subdivide $[-\rho-1,\rho+1]\times I$ into squares $R_{j}$ of side length $1$ then on each $R_{j}$
\[
\|(e^{-\delta|s|}\beta_{0}\cdot\bar\partial_{J} u)|_{R_{j}}\|_{L^{p}}^{p}\le (C'\|(e^{-\delta|s|}\beta_{0}\bar\partial_{J} u)|_{R_{j}}\|_{1})^{p},
\]
summing then gives
\begin{align*}
{\|\beta_{0}\cdot\bar\partial_{J} u\|_{L^{p}_{-\delta}}}^{p} \ &= \ 	
{\|(e^{-\delta|s|}\beta_{0}\cdot\bar\partial_{J} u)|_{R_{j}}\|_{L^{p}}}^{p}\\
&\le \ \sum_{j} (C'\|(e^{-\delta|s|}\beta_{0}\bar\partial_{J} u)|_{R_{j}}\|_{1}^{2})^{\frac{p}{2}} \ \le \ C(\|\beta_{0}\bar\partial_{J} u\|_{1;-\delta}^{2})^{\frac{p}{2}}
\ = \ C{\|\beta_{0}\bar\partial_{J} u\|_{1;-\delta}}^{p}. 
\end{align*}


The second term in the right hand side of \eqref{eq:1norm(u-b)} is supported in $[-\rho-\frac78,-\rho-\frac34]\times I\cup [\rho+\frac34,\rho+\frac78]\times I$ and is of size $\mathcal{O}(e^{-\delta\rho} \epsilon)$.
Therefore, 
we have
\begin{equation}\label{eq:firstonu} 
	\|\beta_{0}\cdot u- b_{\beta_{0}\cdot u}\|_{W^{1,p}_{-\delta}}=\mathcal{O}(\epsilon e^{-\delta\rho}).  
\end{equation}

The remainder of the argument proceeds similarly, with two more invocations of 
the elliptic estimate to control derivatives.  We now proceed in detail.

Let $u=u|_{[-\rho-\frac34,\rho+\frac34]\times I}$ and take  
$\beta_{1}\colon [-\rho-\frac34,\rho+\frac34]\times I\to [0,1]$, a cut off function such that
\[ 
\beta_{1}(s,t)=
\begin{cases}
	0, & (s,t)\in [-\rho-\frac34,-\rho-\frac58]\times I\cup [\rho+\frac58,\rho+\frac34]\times I,\\
	1, & (s,t)\in [-\rho-\frac12,\rho+\frac12]\times I.
\end{cases}
\] 
The linear operator $G_{1}$ on functions $w\colon \R\times I\to\R^{2n}$ is
\[ 
G_1w=\bar\partial w + (\beta_{1}\cdot A)\circ dw \circ i.
\]
We want to show that $G_{1}\colon W^{2,4}_{-\delta}\to W^{1,4}_{-\delta}$ is at distance $\mathcal{O}(\epsilon)$ in operator norm from $\bar\partial$. Observing that the $L^{4}_{-\delta}$ norm of $\beta_{1}\cdot A\circ dw$ is $\mathcal{O}(\epsilon)\|dw\|_{L^{4}_{-\delta}}$, it remains to show that the $L^{4}_{-\delta}$ norm of the derivative of $\beta_{1}\cdot A\cdot dw$ is $\mathcal{O}(\epsilon)\|w\|_{W^{2,4}_{-\delta}}$. This is straightforward for terms where the derivarive lands on $\beta_{1}$ or $dw$. The remaining term, where the derivative lands on $A$, has the form $\beta_{1}\cdot dA\cdot du\cdot dw$. Here $dA$ is bounded and by the Cauchy-Schwarz inequality,  
\[ 
\left(\int \beta_{1}^{4}|du\cdot dw|^{4} e^{-4\delta|s|} dsdt\right)^{\frac14} \le \left(\int \beta_{1}^{8}|du|^{8}\right)^{\frac18}\left(\int |dw|^{8} e^{-8\delta|s|} dsdt\right)^{\frac18}. 
\] 

The first factor is $\mathcal{O}(\epsilon)$ by \eqref{eq:firstonu} for $p=8$
(note the support of $\beta_{1}$ is contained in the region where $\beta_{0}=1$).

The second factor is bounded by $\|dw\|_{W^{1,4}_{-\delta}}$ by the following argument. Note that $W^{1,4}$ sits in $C^{0}$, $\|dw\|_{C^{0}}\le C\|dw\|_{W^{1,4}}$, by Morrey's theorem. Subdivide the region  into length $1$ pieces $R_{j}$. We have on each piece
\[ 
\||dw|\cdot e^{-\delta|s|}|_{R_{j}}\|_{C^{0}}\le C\||dw|e^{-\delta|s|}|_{R_{j}}\|_{W^{1,4}}.
\]    
Then
\[ 
\int_{R_{j}} |dw|^{8} e^{-8\delta|s|} dsds\le C^{8}\||dw|e^{-\delta|s|}|_{R_{j}}\|_{W^{1,4}}^{8}
\]
and summing over $j$ gives the desired bound.

Since $G_{1}$ is $\mathcal{O}(\epsilon)$-close to $\bar\partial$, we conclude that there an estimate
\[ 
\|w-b_{w}\|_{W^{2,4}_{-\delta}}\le C_{2}\|G_1w\|_{W^{1,4}_{-\delta}}.
\]
We then apply this to $\beta_{1}\cdot u$ and argue as above to conclude that 
\begin{equation}\label{eq:secondonu}
	\|\beta_{1}\cdot u-b_{\beta_{1}\cdot u}\|_{W^{2,4}_{-\delta}}=\mathcal{O}(\epsilon).
\end{equation}

The final bootstrap is similar, let $u=u|_{[-\rho-\frac12,\rho+\frac12]\times I}$ and take  
$\beta_{2}\colon [-\rho-\frac12,\rho+\frac12]\times I\to [0,1]$, a cut off function such that
\[ 
\beta_{2}(s,t)=
\begin{cases}
	0, & (s,t)\in [-\rho-\frac12,-\rho-\frac38]\times I\cup [\rho+\frac38,\rho+\frac12]\times I,\\
	1, & (s,t)\in [-\rho-\frac14,\rho+\frac14]\times I.
\end{cases}
\] 
The linear operator $G_{2}$ on functions $w\colon \R\times I\to\R^{2n}$ is
\[ 
G_2w=\bar\partial w + (\beta_{2}\cdot A)\circ dw \circ i.
\]
We want to show that $G_{2}\colon W^{3,2}_{-\delta}\to W^{2,2}_{-\delta}$ is at distance $\mathcal{O}(\epsilon)$ in operator norm from $\bar\partial$. As above most terms in derivatives of $\beta_{2}\cdot A\cdot dw$ are easily estimated. The terms that requires attention are 
\[ 
\beta_{2}\cdot d^{2}A\cdot du\cdot du\cdot dw\qquad\text{and}\qquad \beta_{2}\cdot dA\cdot d^{2}u\cdot dw.
\]
Here, as above we use Cauchy-Schwarz and the estimates $\|du\|_{L^{8}_{\delta}}=\mathcal{O}(\epsilon)$ and $\|d^{2}u\|_{L^{4}_{\delta}}=\mathcal{O}(\epsilon)$ from \eqref{eq:secondonu}. Since $G_{2}$ is close to $\bar\partial$ we get the estimate
\begin{equation}\label{eq:3norm(u-b)} 
	\|w-b\|_{3, -\delta}\le C\|G_{2}w\|_{2, -\delta}
\end{equation}
Applying \eqref{eq:3norm(u-b)} to $\beta_{2}\cdot u$ then gives, for $c=b_{\beta_{1}\cdot u}$, by repetition of the argument above: 
\begin{align*}
	\|(u-c)|_{[-\rho,\rho]\times I}\|_{3,\delta}^{\;\wedge}&\le e^{\delta\rho}\|\beta_{2}\cdot u-b_{\beta_{2}\cdot u}\|_{3,-\delta}\le Ce^{\delta\rho}\|G_{2}\beta_{2}\cdot u\|_{2,-\delta}\\
	&\le C(\|\bar\partial_{J} u\|_{2, \delta}^{\;\wedge}+\|(u-c)|_{[-\rho-1,-\rho]\times I}\|_{2}+\|(u-c)|_{[\rho,\rho+1]\times I}\|_{2}).
\end{align*}
\end{proof}

\begin{remark}
In case $J=J_{\mathrm{std}}$, Proposition \ref{prop: constant control} holds without the conditions on the $C^{0}$-norm of $u$ and on the norm of $\bar\partial_{J}u$. To see this, note that these conditions are used only to guarantee that $A$ is small and if $J=J_{\mathrm{std}}$ then $A=0$.
\end{remark}

\begin{corollary} \label{neck decay from 2 norms *} 
	A nodal degeneration of cylinders $u_\alpha \rightsquigarrow (u_+, u_-)$ satisfying  
		\begin{align}
			\label{limit compatible 2 *} 
			&\|(\bar\partial_{J}u_{\alpha,+} \ - \ \xi_{\alpha,+}\otimes d\bar z)|_{[- \rho_\alpha+1, 0] \times I}\|_{2,2\pi + \delta'} \\\notag
			&\qquad\quad+ 
			\|(\bar\partial_{J} u_{\alpha,-} \ - \ \xi_{\alpha,-}\otimes d\bar z) |_{[0, \rho_\alpha-1] \times I}\|_{2,2\pi + \delta'} = \mathcal{O}(1).
		\end{align}  
		and, for some $\delta > \pi$, 
		\begin{equation}
			\label{cut-off decay2 *}
			\|\bar\partial_{J} u_{\alpha}|_{[-1,1]\times I}\|_{2} = \mathcal{O}(e^{-\delta \rho_{\alpha}}).
		\end{equation}
	has exponential neck decay.  
\end{corollary} 

\begin{proof}
	 By \eqref{limit compatible 2 *} and \eqref{cut-off decay2 *},  $\|\bar\partial u_{\alpha}\|_{1,-\delta}=e^{-\delta\rho}\|\bar\partial u_{\alpha}\|_{1,\delta}^{\;\wedge}=\mathcal{O}(e^{-\delta\rho})$, and $|u_{\alpha}|_{C^{0}}<\epsilon$ since $u_{\alpha}$ maps into $B(\epsilon)$.
	Thus we may apply Proposition \ref{prop: constant control}; call the resulting constant $c_\alpha$.  Then we have: 
	\begin{align}\label{eq:untranslatedLp}
		& \|(u_{\alpha}-c_{\alpha})|_{[-\rho_{\alpha}-\rho_0-1,\rho_{\alpha}+\rho_0 +1]\times I}\|_{3,\delta}^{\;\wedge} \le 
		\\ & \qquad \notag C(\|\bar\partial_{J} u_{\alpha}\|_{2, \delta}^{\;\wedge} +\|(u_{\alpha}-c_{\alpha})|_{[-\rho_{\alpha}-\rho_0-1,-\rho_{\alpha}-\rho_0]\times I}\|_{2}+\|(u_\alpha-c_\alpha)|_{[\rho_{\alpha}+\rho_0,\rho_{\alpha}+\rho_0+1]\times I}\|_{2}).
	\end{align}
	The last two terms are $\cO(1)$ because in the region in question $u_\alpha$ is converging to the fixed
	corresponding part of $u_+$ or $u_-$.  The first term is controlled by \eqref{limit compatible 2 *} and \eqref{cut-off decay2 *}.  
	Finally, the integrating the above estimate shows $\|c_\alpha - c_{\alpha, 0}(0)\|_{3, \delta}^{\; \wedge} = \cO (1)$, so we may replace
	$c_\alpha$ by $c_{\alpha, 0}(0)$ and retain the desired estimate. 
\end{proof}

\section{Perturbations ensuring admissibility} \label{hwz admissible} 

In this section we identify a natural class of perturbations $\lambda(u)$ such that
degenerating sequences with $\bar \partial_J u = \lambda(u)$ will always satisfy the admissibility conditions in Definition \ref{admissible}. 

We fix a formula for pregluing.  Fix 
a smooth cut off function $\beta: \R \to [0, 1]$ equal to $1$ on $[1,\infty)$ and equal to $0$ on $(-\infty,-1]$.
 Consider some fixed $H^{3}_{\delta}$-maps, with common asymptotic constant that we take to be $0\in\R^{2n}$, 
\[ 
u_{+}\colon(-\infty,0]\times I\to\R^{2n},\qquad u_{-}\colon[0,\infty)\times I\to\R^{2n}. 
\]

To any 
\[ 
(v_+, v_-, c, \rho) \in H^{3}_{\delta}((-\infty,0]\times I,\R^{2n})\times H^{3}_{\delta}([0,\infty)\times I,\R^{2n})\times\R^{n}\times [\rho_{0},\infty), 
\]
we associate the map 
$w\colon [-\rho,\rho]\times I\to \R^{2n}$ given by
\[
w(s,t)=\beta(s)(u_{+}(s-\rho,t) + v_{+}(s-\rho,t)) \ + \ (1-\beta(s))(u_{-}(s+\rho,t)+v_{-}(s+\rho,t)) \ + \ c,
\]

It is also natural to interpret $(v_{+},v_{-},c,\infty)$ as the nodal map
\[ 
(w_{+},w_{-})=(u_{+}+v_{+}+c,u_{-}+v_{-}+c)
\] 

\begin{lemma} \label{ndc from polyfold chart}
If $(v_{\alpha, +}, v_{\alpha, -}, c_\alpha, \rho_\alpha) \to (v_+, v_-, c, \infty) $ then the corresponding $w_\alpha \rightsquigarrow (w_+, w_-)$ is a nodal degeneration of cylinders. 
\end{lemma} 

We next discuss a specific form of right hand side for the Cauchy-Riemann equation, $\bar\partial_{J} u=\lambda(u)$. Fix compactly supported smooth 
vector fields $\xi_{\pm}\colon\R^{2n}\to\R^{2n}$.  Fix cutoff functions  $\gamma_{\pm}\colon [-\rho,\rho]\times I\to[0,1]$ with the following properties:
\newpage
\begin{itemize}
	\item[$(+)$] $\gamma_{+}$ equals $1$ on $[1,\rho]\times I$, $\gamma_{+}$ equals $0$ on $[-\rho,0]\times I$, and the derivative $d\gamma_{+}$ is supported in $[1-\frac{1}{\rho},1]\times I$.
	\item[$(-)$] $\gamma_{-}$ equals $1$ on $[-\rho,-\rho+1]\times I$, $\gamma_{-}$ equals $0$ on $[0,\rho]\times I$, and the derivative $d\gamma_{-}$ is supported in $[-1,-1+\frac{1}{\rho}]\times I$.
\end{itemize}

For maps $w$ as constructed above, we define
\[ 
\lambda(w) = \gamma_{+}(s,t) \xi_{+}(w(s,t))\otimes d\bar z + \gamma_{-}(s,t)\xi_{-}\otimes d\bar z,
\]

In the nodal case $(w_{+},w_{-})$, we have
\begin{align*}
	\lambda(w_{+}) &=\xi_{+}(w_{+})\otimes d\bar z = \xi_{+}(w_{+}(s,t))e^{2\pi(s+it)}(ds-idt),\\
	\lambda(w_{-}) &=\xi_{-}(w_{-})\otimes d\bar z = \xi_{-}(w_{-}(s,t))e^{-2\pi(s+it)}(ds-idt),
\end{align*} 

\begin{proposition}\label{prp: perturbation} In the situation of Lemma \ref{ndc from polyfold chart}, 
suppose additionally that  $\bar\partial_{J} w_{\alpha}=\lambda(w_{\alpha})$ with $\lambda$ as above.  
Then $w_{\alpha}\rightsquigarrow (w_{+},w_{-})$ satisfies \eqref{cut-off decay}, \eqref{limit compatible 2}, and \eqref{cut-off decay2}.
\end{proposition}

\begin{proof}
	If $c_{\alpha}$ is the asymptotic constant of $w_{\alpha}$ then by Taylor expansion of the smooth vector field $\xi_{+}$:
	\begin{align*} 
		\|(\bar\partial_{J} w_{\alpha}-\xi_{+}(c_{\alpha})\otimes d\bar z)|_{[1,\rho]\times I}\|_{2,2\pi+\delta}^{\;\wedge} &=
		\|(\xi_{+}(w_{\alpha})-\xi_{+}(c_{\alpha}))e^{2\pi(s+it)}|_{[1,\rho]\times I}\|_{2,2\pi+\delta}^{\;\wedge}\\
		&=\mathcal{O}(\|w_{\alpha}\|_{2,\delta}^{\;\wedge})=\mathcal{O}(1).
	\end{align*}
	The argument for other half of $[-\rho,\rho]\times I$ is identical and we find that \eqref{limit compatible 2} holds. 
	
	Also,
	\begin{align*}
		\|\bar\partial_{J} w_{\alpha}|_{[-1,1]\times I}\|_{0,2\pi}^{\;\wedge} &= \left(\int_{[0,1]\times I}|\gamma_{+}(s,t)\xi_{+}(w_{\alpha})|^{2}\right)^{\frac12}+\left(\int_{[-1,0]\times I}|\gamma_{-}(s,t)\xi_{-}(w_{\alpha})|^{2}\right)^{\frac12}\\
		&=\mathcal{O}(1/\sqrt{\rho}),
	\end{align*} 
	and \eqref{cut-off decay} holds. Finally, to see that \eqref{cut-off decay2} holds, we Taylor expand $\gamma_{\pm}$ and note that we can take $d^{(k)}\gamma_{\pm}=\mathcal{O}(\rho^{-k})$ and therefore for $1\le k\le 2$:
	\begin{align*} 
		\|d^{(k)}\bar\partial_{J} w_{\alpha})|_{[-1,1]\times I}\|_{0,\delta}^{\;\wedge} &= 
		\mathcal{O}(\|d^{(k)}\gamma_{\pm}(w_{\alpha})\otimes d\bar z|_{[-1,1]\times I}\|_{0,\delta}^{\;\wedge})\\ 
		&=\mathcal{O}(\rho^{2})e^{-(2\pi+\delta)\rho}\|(w_{\alpha}-c_{\alpha})|_{[-1,1]\times I}\|_{2,\delta}=\mathcal{O}(1).
	\end{align*}
\end{proof}

\begin{remark}
The space of $(v_+, v_-, c, \rho)$ is used in \cite{HWZ-GW} to provide charts, hence define the  topology, on
configuration spaces of maps.  Thus the sort of perturbations described above are natural in that context. 
\end{remark} 


\bibliographystyle{hplain}
\bibliography{skeinrefs}

\end{document}